\documentclass[11pt, reqno]{amsart}

\usepackage{etoolbox}

\newbool{ForSubmission}
\booltrue{ForSubmission}

\usepackage{amscd,amsmath}
\usepackage{mathrsfs}
\usepackage{amsfonts}
\usepackage{amssymb, bm, xspace}
\usepackage{enumerate}
\usepackage{setspace}

\usepackage{datetime}

\usepackage{enumitem}

\usepackage{xcolor}

\usepackage[pagebackref=true, colorlinks=true, citecolor=blue]{hyperref}

\usepackage[margin=1.2in]{geometry}

\usepackage{mathtools}

\usepackage{cleveref}

\usepackage{stackengine}
\stackMath

\newcommand{\Comment}[1]{{\color{blue}#1}}
\newcommand{\OptionalDetails}[1]{
    \ifbool{ForSubmission}
        {
        }
        {\begin{quote}\Comment{\footnotesize
        \medskip

        \noindent#1}
        \end{quote}
        }
    }

\newbool{HaveBBM}
\booltrue{HaveBBM}

\ifbool{HaveBBM}{
	\usepackage{bbm}
    }
    {
    }

\newbool{arXivFormat}
\boolfalse{arXivFormat}
    
\newcommand{\IfarXivElse}[2]{
    \ifbool{arXivFormat}
        {#1}{#2}
    }

\renewcommand{\mathbf}[1]{\bm{#1} \textbf{ *** Use bm instead of mathbf ***}}

\newcommand{\eqn}{\begin{eqnarray}}
\newcommand{\een}{\end{eqnarray}}

\newtheorem{theorem}{Theorem}[section]
\newtheorem*{theorem*}{Theorem}				
\newtheorem{prop}[theorem]{Proposition}
\newtheorem{lemma}[theorem]{Lemma}
\newtheorem{cor}[theorem]{Corollary}

\newtheorem{definition}[theorem]{Definition}
\newtheorem{remark}[theorem]{Remark}

\numberwithin{equation}{section}

\newcommand{\abs}[1]{\left\vert#1\right\vert}

\newcommand{\innp}[1]{\ensuremath{\left< #1 \right>}}

\newcommand{\BB}[1]{\ensuremath{\mathbb{#1}}}

\newcommand{\R}{\ensuremath{\BB{R}}}
\newcommand{\N}{\ensuremath{\BB{N}}}
\newcommand{\Z}{\ensuremath{\BB{Z}}}
\newcommand{\iny}{\ensuremath{\infty}}
\newcommand{\grad}{\ensuremath{\nabla}}

\newcommand{\CharFunc}{
    \ifbool{HaveBBM}{
        \ensuremath{\mathbbm{1}}
        }
        {
        \ensuremath{\bm{1}}
        }
    }

\DeclareMathOperator{\dv}{div} %
\DeclareMathOperator{\curl}{curl} %
\DeclareMathOperator{\tr}{tr} %
\DeclareMathOperator{\cofac}{cofac} %
\DeclareMathOperator{\spn}{span} %
\DeclareMathOperator{\proj}{proj} %
\DeclareMathOperator{\supp}{supp} %
\DeclareMathOperator{\PV}{p.v.} %
\newcommand{\prt}{\ensuremath{\partial}}
\newcommand{\brac}[1]{\ensuremath{\left[ #1 \right]}}
\newcommand{\bigbrac}[1]{\ensuremath{\Big[ #1 \Big]}}
\newcommand{\pr}[1]{\ensuremath{\left( #1 \right)}}
\newcommand{\bigpr}[1]{\ensuremath{\Big( #1 \Big)}}
\DeclarePairedDelimiter{\set}{\{}{\}}

\newcommand{\bigset}[1]{\ensuremath{\left\{ #1 \right\}}}

\newcommand{\norm}[1]{\ensuremath{\left\Vert #1 \right\Vert}}

\newcommand{\smallnorm}[1]{\ensuremath{\Vert #1 \Vert}}

\newcommand\tenq[2][1]{%
	\def\useanchorwidth{T}%
	\ifnum#1>1%
		\stackunder[0pt]{\tenq[\numexpr#1-1\relax]{#2}}{\scriptscriptstyle\sim}%
	\else%
		\stackunder[1pt]{#2}{\scriptscriptstyle\sim}%
	\fi%
	}

\renewcommand{\epsilon}{\varepsilon}
\newcommand{\eps}{\ensuremath{\varepsilon}}

\newcommand{\Cal}[1]{\ensuremath{\mathcal{#1}}}
\newcommand{\al}{\ensuremath{\alpha}}
\newcommand{\la}{\ensuremath{\lambda}}

\newcommand{\diff}[2]{\frac{ d#1}{d#2}}

\newcommand{\ul}{\underline}
\newcommand{\smallabs}[1]{\ensuremath{\vert #1 \vert}}

\renewcommand{\matrix}[2]{\begin{pmatrix} #1 \\ #2 \end{pmatrix}}

\newcommand{\matrixthree}[3]{\begin{pmatrix} #1 \\ #2 \\ #3 \end{pmatrix}}
\newcommand{\spaciousmatrixthree}[3]{\begin{pmatrix} #1 \\ \\ #2 \\ \\ #3 \end{pmatrix}}
\newcommand{\threevector}[3]{\matrixthree{#1}{#2}{#3}}
\newcommand{\spaciousthreevector}[3]{\spaciousmatrixthree{#1}{#2}{#3}}

\newcommand{\Holder}
    {H\"{o}lder\xspace}

\newcommand{\Ignore}[1]{}

\definecolor{Correction}{named}{red}

\newcommand{\spacer}{\vspace{2mm}}
\newcommand{\halfspacer}{\vspace{1mm}}

\crefname{cor}{Corollary}{Corollaries} 
									   
\crefname{lemma}{Lemma}{Lemmas}	       

\crefname{section}{Section}{Sections}
\Crefname{section}{Section}{Sections}

\crefname{appendix}{Appendix}{Appendices}
\Crefname{appendix}{Appendix}{Appendices}

\crefname{theorem}{Theorem}{Theorems}
\Crefname{theorem}{Theorem}{Theorems}

\crefname{prop}{Proposition}{Propositions}
\Crefname{prop}{Proposition}{Propositions}

\crefname{conj}{Conjecture}{Conjectures}
\Crefname{conj}{Conjecture}{Conjectures}

\crefname{definition}{Definition}{Definitions}
\Crefname{definition}{Definition}{Definitions}

\crefname{remark}{Remark}{Remarks}
\Crefname{remark}{Remark}{Remarks}

\crefname{assumption}{Assumption}{Assumptions}
\Crefname{assumption}{Assumption}{Assumptions}

\crefformat{equation}{(#2#1#3)}
\crefrangeformat{equation}{(#3#1#4) through (#5#2#6)}
\crefmultiformat{equation}
    {(#2#1#3)}%
    { and~(#2#1#3)}
    {, (#2#1#3)}
    { and~(#2#1#3)}

\begin{document}

\newdateformat{mydate}{\THEDAY~\monthname~\THEYEAR}

\title
    [Striated Regularity for the Euler Equations]
    {Striated Regularity for the Euler Equations}

\author{Hantaek Bae}
\address{Department of Mathematical Sciences, Ulsan National Institute of Science and Technology (UNIST), Korea}
\email{hantaek@unist.ac.kr}

\author{James P Kelliher}
\address{Department of Mathematics, University of California, Riverside, USA}
\email{kelliher@math.ucr.edu}


\Ignore{ 
\begin{abstract}
The well-posedness of the Euler equations in \Holder spaces for short time in 3D goes back to the work of Gunther and Lichtenstein in the 1920s; the global-in-time 2D result is due to Wolibner in 1933. The work in 2D of Chemin and in higher dimensions of Gamblin and Saint Raymond, and of Danchin, in the 1990s established analogous results for vorticity possessing negative \Holder space regularity only in directions given by a sufficient family of vector fields, which are themselves pushed-forward by the flow (\emph{striated} regularity).
We give an alternative proof of these results and, by relating the striated regularity of the velocity and vorticity, we establish the propagation of striated regularity of the Lagrangian velocity. 
Finally, we show in 2D and 3D that the velocity gradient is regular after being corrected by a regular matrix multiple of the vorticity.
\end{abstract}
} 

\begin{abstract}
In 1993, Chemin proved that vorticity possessing negative \Holder regularity in directions given by a sufficient family of vector fields (striated regularity) maintains such regularity for all time when measured against the push-forward of those vector fields.
Later work of Gamblin and Saint Raymond, and of Danchin, established analogous results in higher dimension.
We give an alternative proof of these results, and establish the propagation of striated regularity of the Lagrangian velocity in a positive \Holder space.

\end{abstract}

\maketitle

\ifbool{ForSubmission}{
	\vspace{-2em}
	\setcounter{tocdepth}{1}    
	\tableofcontents
    }
    {
    \begin{center}
    {
        \footnotesize
        \Comment{
        \framebox[0.69\textwidth]{
            \parbox[center][3.5em][c]{0.66\textwidth}
            {
            Includes some proofs and additional details not intended
            for submission.
            These details appear in blue in small font.           
        }
        }
        }
    }
    \end{center}

    \renewcommand\contentsname{}  
    \setcounter{tocdepth}{1}      
    {\small
        \tableofcontents
    }
    }

\vspace{-2em}

\section{Introduction}\label{S:Introduction}

\noindent The Euler equations in velocity form (without forcing) on $\R^d$, $d \ge 2$, can be written,
\begin{align}\label{e:Eu}
    \left\{
    \begin{array}{rl}
        \prt_t u + (u \cdot \nabla) u + \nabla p &= 0, \\
        \dv u &= 0, \\
        u(0) &= u_0,
    \end{array}
    \right.
\end{align}
where $u$ is the velocity field, $p$ is the pressure, and $u_0$ is the divergence-free initial velocity.
These equations model the flow of an incompressible inviscid fluid.

\smallskip
\begin{center}
    \fbox{Throughout this paper we fix $\al \in (0, 1)$.}
\end{center}
\smallskip

The fundamental well-posedness (though not in the sense of Hadamard) result in \Holder spaces is given in the following theorem:

\begin{theorem*}[Lichtenstein 1925, 1927, 1928; Gunther 1927, 1928; Wolibner 1933]
	Assume that $u_0 \in C^{1, \al}(\R^d)$, $d = 3$. There exists a unique
	solution to the Euler equations with
	$u \in L^\iny(0, T; C^{1, \al})$ for some $T > 0$. When $d = 2$,
	$T$ can be taken arbitrarily large.
\end{theorem*}

The 3D result goes back to papers of Lichtenstein and Gunther 
\cite{Lichtenstein1925,Lichtenstein1927,Lichtenstein1928,Lichtenstein1930,Gunter1927A,Gunter1927B,Gunter1928}, the 2D result is due to Wolibner \cite{Wolibner1933}. We mention also Chemin's proof in \cite{C1998}.

In this paper, we will show that, in fact, such well-posedness can be obtained assuming $C^{1, \al}$ regularity of the velocity only in directions given by a sufficient family of vector fields. To describe this result, we first need to review the vorticity formulation of the Euler equations, introduce the flow map associated to the Eulerian velocity along with the pushforward of a velocity field by the flow map, and define some function spaces on families of vector fields.

We define the vorticity in any of three different ways as follows:
\begin{align}\label{e:VorticityDef}
	\begin{array}{rl}
		\halfspacer
		d = 2:
			&\omega
    			= \omega(u)
				:= \prt_1 u^2 - \prt_2 u^1, \\
		d = 3:
			&\vec{\omega}
				= \vec{\omega}(u)
				:= \curl u, \\
		d \ge 2:
			& \Omega
				= \Omega(u)
				:= \displaystyle \grad u - (\grad u)^T; \\
			& \Omega^j_k = \prt_k u^j - \prt_j u^k.
	\end{array}
\end{align}

When working exclusively in 2D, it is always most convenient to use the first definition. Even specialized to 3D, most of our computations are more easily accomplished using the third definition than the second. When we express results or give proofs that apply to all dimensions $d \ge 2$ we will use the third form; when specializing to 2D we will use the first. A similar comment applies to the expressions that appear below in \cref{e:VorticityE,e:BSLaw}, \cref{P:graduExp,P:graduYLikeLemma}, and \cref{C:graduCor}.

Taking the vorticity of $\cref{e:Eu}_1$, we obtain the vorticity equations,
\begin{align}\label{e:VorticityE}
	\begin{array}{rl}
		\halfspacer
		d = 2:
			&\prt_t \omega + u \cdot \grad \omega = 0, \\
		d = 3:
			&\prt_t \vec{\omega} + u \cdot \grad \vec{\omega} = \vec{\omega} \cdot \grad u, \\
		d \ge 2:
			&\prt_t \Omega + u \cdot \grad \Omega + \Omega \cdot \grad u = 0.
	\end{array}
\end{align}

To turn \cref{e:VorticityE} into a vorticity formulation, the velocity is recovered from the vorticity using the Biot-Savart law. Letting $\Cal{F}_d$ be the fundamental solution of the Laplacian in $\R^d$ ($\Delta \Cal{F}_d = \delta$), we can write this as
\begin{align}\label{e:BSLaw}
	\begin{array}{rll}
		\halfspacer
		d = 2:
			&u = K * \omega,
			&K := \grad^\perp \Cal{F}_2
				:= (-\prt_2 \Cal{F}_2, \prt_1 \Cal{F}_2),\\
		d \ge 2:
			&u^j = K_d^k * \Omega^j_k,
			&K_d := \grad \Cal{F}_d,
	\end{array}
\end{align}
where here and in all that follows we implicitly sum over repeated indices. For $d = 2$, $3$,
\begin{align}\label{e:BS2D}
	\begin{array}{lll}
		\displaystyle
		\Cal{F}_2(x)
			= \frac{1}{2\pi}\log \abs{x},
	    &\displaystyle
	    K_2(x)
    	    = \frac{1}{2 \pi} \frac{x}{\abs{x}^2},
	    &\displaystyle
	    K(x)
    	    = \frac{1}{2 \pi} \frac{x^\perp}{\abs{x}^2}, \\
        \\
		\displaystyle
		\Cal{F}_3(x)
			= -\frac{1}{4\pi \abs{x}},
	    &\displaystyle
	    K_3(x)
    	    = \frac{1}{4 \pi} \frac{x}{\abs{x}^3},
	\end{array}
\end{align}
where $x^\perp := (-x_2, x_1)$.

Suppose that $u$ is sufficiently regular that it has a unique associated flow map $\eta$,
\begin{align}\label{e:etaDef}
    \prt_t \eta(t, x)
        = u \pr{t, \eta(t, x)},
            \quad
            \eta(0, x) = x.
\end{align}
Let $Y_0$ be a vector field on $\R^d$ and define the pushforward of $Y_0$ by
\begin{align}\label{e:PushForward}
    Y(t, \eta(t, x)) := (Y_0(x) \cdot \nabla) \eta(t, x).
\end{align}
This is just the Jacobian of the diffeomorphism $\eta(t, \cdot)$ multiplied by $Y_0$. Equivalently,
\begin{align*}
    Y(t, x)
        = \eta(t)_* Y_0(t, x)
        := (Y_0(\eta^{-1}(t, x))
            \cdot \grad) \eta(t, \eta^{-1}(t, x)).
\end{align*}

For a $d \times d$ matrix $M$, let $\cofac M$ be its cofactor matrix; thus, $(\cofac M)^i_j = (-1)^{i + j}$ times the $(i, j)$-minor of $M$ (the determinant of the $(d - 1) \times (d - 1)$ matrix formed by removing the $i$-th row and $j$-th column). For any $Y_1, \cdots, Y_{d - 1}$ in $\R^d$, we define $\wedge_{i < d} Y_i$ to be the vector, $Z$, appearing in the last column of the cofactor matrix,
\begin{align*}
	\cofac
	\begin{pmatrix}
		Y^1 & Y^2 \ \cdots \ Y^{d - 1} & Z
	\end{pmatrix}.
\end{align*}
Note that the last column of this cofactor matrix depends only upon $Y^1, \dots, Y^{d - 1}$, so this uniquely defines $Z$.
(There are various equivalent ways to define $\wedge_{i < d} Y_i$, as, for instance, in \cite{Danchin1999}.) Specifically, in 2D and 3D,
\begin{align*}
	\begin{array}{ll}
		\displaystyle
		\wedge_{i < 2} Y_i
			= Y_1^\perp
			:= (-Y_1^2, Y_1^1), & d = 2, \\
		\displaystyle
		\wedge_{i < 3} Y_i
			= Y_1 \times Y_2, & d = 3.
	\end{array}
\end{align*}

Let $\Cal{Y} = (Y^{(\la)})_{\la \in \Lambda}$ be a family of vector fields on $\R^d$ indexed over the set $\Lambda$. For any function $f$ on vector fields (such as $\dv$), define 
\begin{align*}
    f(\Cal{Y}) := \pr{f(Y^{(\la)})}_{\la \in \Lambda}.
\end{align*}
For example, if $\Omega$ is as in \cref{e:VorticityDef}, then for $j, k = 1, \dots, d$,
\begin{align*} 
	\dv (\Omega^j_k \Cal{Y})
		= \pr{\dv \pr{\Omega^j_k (Y^{(\la)})}}_{\la \in \Lambda}.
\end{align*}

For any Banach space, $X$, define
\begin{align*}
    \norm{f(\Cal{Y})}_X
        &:= \sup_{\la \in \Lambda} \norm{f \pr{Y^{(\la)}}}_X.
\end{align*}
When $\norm{f(\Cal{Y})}_X < \iny$ we say that $f(\Cal{Y}) \in X$. Define,
\begin{align}\label{e:IY0}
	\begin{split}
	\begin{array}{rl}
		d = 2:
			& \displaystyle
			I(\Cal{Y})
        		:= \inf_{x \in \R^2} \sup_{\la \in \Lambda} \abs{Y^{(\la)}(x)}, \\
		d \ge 2:
			& \displaystyle
			I(\Cal{Y})
        		:= \min \bigset{\inf_{x \in \R^d} \sup_{\la \in \Lambda} \abs{Y^{(\la)}(x)},
					\inf_{x \in \R^d} \sup_{\la_1, \dots, \la_{d - 1} \in \Lambda}
					\abs{\wedge_{j < d} Y^{(\la_j)}(x)}}.
	\end{array}
	\end{split}
\end{align}

We define the pushforward, $\Cal{Y}$, of the family $\Cal{Y}_0$ by
\begin{align}\label{e:YFamily}
    \Cal{Y}(t) = (Y^{(\la)}(t))_{\la \in \Lambda}, \quad
        Y^{(\la)}(t, \eta(t, x)) := (Y_0^{(\la)}(x) \cdot \nabla) \eta(t, x).
\end{align}

We call $\Cal{Y}_0$ a \textit{sufficient} $C^\al$ family of vector fields if
\begin{align*}
	\Cal{Y}_0 \in C^\al, \;\; \dv \Cal{Y}_0 \in C^\al, \;\; \text{ and } I(\Cal{Y}_0) > 0.
\end{align*}
We will see that the pushforward, $\Cal{Y}(t)$, of $\Cal{Y}_0$ will remain a sufficient family for all time for $d = 2$ and for short time for $d \ge 3$, though the bound on $I(\Cal{Y}(t))$ will increase with time.

We can now state our main results, \cref{T:FamilyVel,T:A,T:Equivalence}. We note that \cref{T:FamilyVel} precisely states the well-posedness of the Euler equations assuming $C^{1, \al}$ regularity of the velocity only in directions given by a sufficient family of vector fields.

\begin{theorem}\label{T:FamilyVel}
	Let $\Cal{Y}_0$ be a sufficient $C^\al$ family of vector fields in $\R^d$, $d \ge 2$.
	Assume that $\Omega(u_0) \in L^1 \cap L^\iny(\R^d)$ and $\Cal{Y}_0 \cdot \grad u_0 \in C^\al$.
	Then for some $T > 0$, $T$ being arbitrarily large when $d = 2$,
	there exists a unique (see \cref{R:Uniqueness}) solution to the Euler equations, with
	$\Cal{Y} \cdot \grad u \in L^\iny(0, T; C^\al)$.
	Moreover, we have the following estimates:
	\begingroup
	\allowdisplaybreaks
	\begin{align}
        &\norm{\grad u(t)}_{L^\iny}
            \le c_2 e^{c_1 t},
            \label{e:MainBoundsgradu} \\
        &\norm{\Cal{Y}(t)}_{C^\alpha}
        	\le c_3 e^{c_1 e^{c_1 t}},
	    	\label{e:MainBoundsY} \\
        &\norm{\dv \Cal{Y}(t)}_{C^\al}
            \le \norm{\dv \Cal{Y}_0}_{C^\al} e^{e^{c_1 t}},
                \label{e:YdivBound}\\
    	&\smallnorm{\dv (\Omega^j_k \Cal{Y})(t)}_{C^{\al - 1}}
    		\le c_3 e^{c_1 e^{c_1 t}}
				\; \forall \; j, k,
    			\label{e:omegaYdivBound} \\
        &\norm{\Cal{Y} \cdot \grad u(t)}_{C^\al}
		\le c_4 e^{c_1 e^{c_1 t}},
			\label{e:YgraduBound} \\
        &\norm{\grad \eta(t)}_{L^\iny}, \, \smallnorm{\grad \eta^{-1}(t)}_{L^\iny}
		\le e^{c_1 e^{c_1 t}},
			\label{e:gradetaMainBound} \\
        &I(\Cal{Y})(t)
        	\ge I(\Cal{Y}_0) e^{-c_1 e^{c_1 t}}.
			\label{e:IYBound}
    \end{align}
    \endgroup
    Here,
    \begin{align*}
			c_1
				:= \frac{C}{\al}, \quad
			c_2
				:= \frac{C}{\al^2}, \quad
			c_3
				:= \frac{C}{\al (1 - \al)}, \quad
			c_4
				:= \frac{C}{\al^2 (1 - \al)}.
	\end{align*}
	The constant $C = C(u_0, \Cal{Y}_0)$ depends on $u_0$ and $\Cal{Y}_0$; specifically,
	on $\norm{\Omega_0}_{L^1 \cap L^\iny}$, $\norm{\Cal{Y}_0 \cdot \grad u_0}_{C^\al}$,
	$\norm{\Cal{Y}_0}_{C^\al}$, $\norm{\dv \Cal{Y}_0}_{C^\al}$,
	$I(\Cal{Y}_0)^{-1}$.
	In each case, $C$ increases with each of these quantities.
	In 3D, $c_1$ through $c_4$ also have an additional dependence on $T$.
\end{theorem}

\begin{theorem}\label{T:A}
	Let $u$ be the solution given by \cref{T:FamilyVel} for $d = 2$.
	There exists a matrix $A(t) \in C^\alpha(\R^2)$ such that for all $t \ge 0$,
    \begin{align}\label{e:ABound}
        \norm{A(t)}_{C^\al}, \,
        \norm{\nabla u(t) - \omega(t) A(t)}_{C^\al}
            &\le c_5 e^{c_1 e^{c_1 t}},
    \end{align}
    where $c_1$ is in \cref{T:FamilyVel} and
	\begin{align*}
		c_5
			:= \frac{C(u_0, \Cal{Y}_0)}{\al^4 (1 - \al)^4}.
    \end{align*}

	When $d = 3$, the same result holds, though now we have $A(t) \Omega(t)$ in place
	of $\omega(t) A(t)$.
	In 3D, $c_5$ also has an additional dependence on $T$.
\end{theorem}

\begin{theorem}\label{T:Equivalence}
	Let $\Cal{Y}$ be a sufficient family of $C^\al$ vector fields and
	$\Omega \in L^1 \cap L^\iny(\R^d)$.
	Then
    \begin{align}\label{e:RegEquivalence}
    	\begin{array}{ll}
            \Cal{Y} \cdot \grad u \in C^\al
            	\iff
            \dv (\omega \Cal{Y}) \in C^{\al - 1}
        &d = 2,  \\
        	\Cal{Y} \cdot \grad u \in C^\al
				\iff
        	\dv (\Omega^j_k \Cal{Y}) \in C^{\al - 1} \, \forall \, j, k,
		&d \ge 2.
		\end{array}
    \end{align}
\end{theorem}

\cref{T:Equivalence} is very close to Lemma 4.6 in Fanelli's \cite{Fanelli2012}, and indeed follows from that lemma combined (for the forward implications) with \cref{P:graduLInf}, below---see \cref{S:Equivalence}. The backward implications in \cref{e:RegEquivalence} are implicit in the proofs in \cite{Chemin1991VortexPatch,Chemin1993Persistance,C1998,SerfatiVortexPatch1994,Danchin1999} (see \cref{R:CheminDanchinVelocity}).

\cref{T:Equivalence} allows us to rephrase our striated regularity results in Lagrangian form. Define the
	\textit{Lagrangian velocity},
	\begin{align*}
		v(t, x)
			:= u(t, \eta(t, x)).
	\end{align*}
	A classical calculation using the chain rule gives, for any $Y_0 \in \Cal{Y}_0$,
	\begin{align*}
		Y_0(x) \cdot \grad v(t, x)
			&= (Y \cdot \grad u)(t, \eta(t, x)).
	\end{align*}
	Thus (see \cref{e:CalphaFacts}),
	\begin{align*}
	\norm{Y_0 \cdot \grad v(t)}_{C^\al}
		&\le \norm{(Y \cdot \grad u)(t)}_{C^\al} \norm{\grad \eta(t)}_{L^\iny}^\al.
	\end{align*}
As a simple corollary of \cref{T:FamilyVel}, then, we see that the striated regularity of the Lagrangian velocity is propagated over time:
\begin{cor}
	Making the assumptions in \cref{T:FamilyVel},
	$\Cal{Y}_0 \cdot \grad v(t)$ remains in $C^\al$ for all time in
	2D and up to time $T$ for $d \ge 3$.
\end{cor}

The equivalence of striated regularity of the initial vorticity and velocity in \cref{T:Equivalence} yields an immediate proof of \cref{T:FamilyVel} when combined with the following two existing results for the propagation of regularity of striated vorticity.

\begin{theorem}\label{T:FamilyVelOmega2D}[Chemin \cite{C1998}]
	Let $\Cal{Y}_0$ be a sufficient $C^\al$ family of vector fields in $\R^2$.
	Assume that $\omega(u_0) \in L^1 \cap L^\iny(\R^2)$ and
	$\dv (\omega_0 \Cal{Y}_0) \in C^{\al - 1}$.
	(The negative \Holder space, $C^{\al - 1}$, is defined in \cref{S:Notation}.)
	Then there exists a unique global solution to the Euler equations, with
	$\dv (\omega(t) \Cal{Y}(t)) \in L^\iny_{loc}([0, \iny); C^{\al - 1})$.
\end{theorem}

\begin{theorem}\label{T:FamilyVelOmega3D}[Danchin \cite{Danchin1999}]
	Let $\Cal{Y}_0$ be a sufficient $C^\al$ family of vector fields in $\R^d$, $d \ge 3$.
	Assume that $\Omega(u_0) \in L^1 \cap L^\iny(\R^d)$ and
	$\dv(\Omega^j_k(u_0) \Cal{Y}_0) \in C^{\al - 1}(\R^d)$ for all $j$, $k$.
	Then for some $T > 0$ there exists a unique solution to the Euler equations, with
	$\dv(\Omega^j_k(u(t)) \Cal{Y}(t)) \in L^\iny(0, T; C^{\al - 1})$
	for all $j$, $k$.
\end{theorem}

\begin{remark}\label{R:CheminDanchinVelocity}
	As part of the proofs of \cref{T:FamilyVelOmega2D,T:FamilyVelOmega3D} in \cite{C1998,Danchin1999},
	it is shown that $\Cal{Y} \cdot \grad u \in L^\iny(0, T; C^\al)$. Some form of all the
	estimates stated in \crefrange{e:MainBoundsgradu}{e:IYBound} are also obtained,
	some implicitly, though the specific dependence on $\al$ is not noted.
\end{remark}

\begin{remark}\label{R:Uniqueness}
	For uniqueness in \cref{T:FamilyVel,T:FamilyVelOmega3D} for $d = 2$
	and in \cref{T:FamilyVelOmega2D},
	the condition that $\omega \in L^\iny(0, T; L^1 \cap L^\iny)$ suffices, by Yudovich \cite{Y1963}.
	For higher dimension, a uniqueness condition that suffices for
	\cref{T:FamilyVel,T:FamilyVelOmega3D} is that $u \in L^\iny(0, T; Lip) \cap C(0, T; H^1)$,
	as established in \cite{Danchin1999}. The family $\Cal{Y}$ itself clearly cannot enter
	into any uniqueness criterion.
\end{remark}

The 2D result of Chemin in \cite{C1998} builds on his work in \cite{Chemin1991VortexPatch,Chemin1993Persistance}, which applies to one vector field. In dimensions 2 and higher they were obtained by Danchin in \cite{Danchin1999} (recently extended to nonhomogeneous incompressible fluids by Fanelli in \cite{Fanelli2012}). See also \cite{GamblinSaintRaymond1995,Serfati3DStratified}.

Later, Serfati in \cite{SerfatiVortexPatch1994} also obtained the equivalent of \cref{T:FamilyVelOmega2D} for one vector field as well as the 2D version of \cref{T:A} for striated vorticity and one vector field.
Serfati's proof is elementary. It is, however, terse to the point of obscurity. (The ``proof,'' for instance, of \cref{T:A} is one-sentence long, and does not extend to 3D.) We give a (nearly in 3D) self-contained proof of Chemin's and Danchin's results inspired by Serfati's approach. What makes the proof only nearly self-contained in 3D is our use of an estimate on vortex stretching in dimensions 3 and higher from \cite{Danchin1999}.
We present the full details in 2D, but just outline what is different in the higher-dimensional argument.

We close this introduction by observing a simple consequence of \cref{T:A}: the local propagation in 2D of \Holder regularity stated in \cref{T:LocalPropagation}.

\begin{theorem}\label{T:LocalPropagation}[2D]
    Let $\omega_0$, $Y_0$
    be as in \cref{T:A}.
    If $\omega_0 \in C^{\alpha}(U)$ for some open subset $U$ of $\R^2$
    and $\alpha \in [0, 1)$ then
    $\omega(t) \in C^\alpha(U)$ for all $t$, with
    \begin{align}\label{e:omegaCalLocalBound}
        \norm{\omega(t)}_{C^\alpha(U_t)}
            \le \norm{\omega_0}_{C^\alpha(U)}
                e^{\alpha c_1 e^{c_1 t}},
    \end{align}
    where $U_t = \eta(t, U)$.
    Further,
    \begin{align}\label{e:graduCalLocalBound}
            \norm{\grad u(t)}_{C^\al(U_t)}
            \le c_5 \pr{1 + \norm{\omega_0}_{C^\al(U)}} e^{c_1 e^{c_1 t}}.
    \end{align}
    The constants $c_1$ and $c_5$ are as in \cref{T:A}.
\end{theorem}
\begin{proof}
For any $x, y \in U_t$,
\begin{align*}
    \frac{\abs{\omega(t, x) - \omega(t, y)}}{\abs{x - y}^\alpha}
        = \frac{\abs{\omega_0(\eta^{-1}(t, x)) - \omega_0(\eta^{-1}(t, y))}}
            {\abs{\eta^{-1}(t, x) - \eta^{-1}(t, y)}^\alpha}
            \pr{\frac{\abs{\eta^{-1}(t, x) - \eta^{-1}(t, y)}}
                {\abs{x - y}}}^\alpha.
\end{align*}
Together with \cref{e:gradetaMainBound} this gives \cref{e:omegaCalLocalBound} (a bound that would hold for any Lipschitz velocity field). The bound in \cref{e:graduCalLocalBound} then follows from \cref{e:ABound}.
\end{proof}

\cref{T:LocalPropagation} improves, for initial data having striated regularity, existing estimates of local propagation of \Holder regularity for bounded initial vorticity. For instance, Proposition 8.3 of \cite{MB2002} would only give $\grad u(t) \in C^\al_{loc}(U_t)$.

\medskip

This paper is organized as follows. In \cref{S:Notation}, we fix some notation and make a few  definitions. We develop the basic estimates we need on singular integrals in \cref{S:SingularIntegrals}. \cref{A:graduCalcs} includes a number of lemmas centered around $\grad u$, these lemmas being central to the proofs of all of our results. Our proofs of \cref{T:FamilyVelOmega2D,T:FamilyVelOmega3D,T:Equivalence} all rely upon a linear algebra lemma of Serfati's to obtain a refined estimate on $\grad u$ in $L^\iny$. We present this lemma in \cref{S:Lemmas}. The proof of \cref{T:Equivalence}, giving the equivalence of striated regularity of velocity and vorticity, is presented in \cref{S:Equivalence}. In \cref{S:MatrixA}, we give the proof of \cref{T:A} in 2D, giving the 3D proof in \cref{S:MatrixA3D}.

With \cref{S:MatrixA3D}, we have a complete proof of our main results: we directly proved  \cref{T:A,T:Equivalence}, and \cref{T:FamilyVel} follows from \cref{T:Equivalence} applied to \cref{T:FamilyVelOmega2D,T:FamilyVelOmega3D} of \cite{Chemin1991VortexPatch,Chemin1993Persistance,C1998,Danchin1999}.
In \cref{S:Transport}, we begin a direct, elementary proof of \cref{T:FamilyVelOmega2D}, inspired by \cite{SerfatiVortexPatch1994}. From this we derive, as well, the specific estimates stated in \crefrange{e:MainBoundsgradu}{e:IYBound}.

The subject of \cref{S:Transport} is the transport equations of a vector field $Y_0 \in \Cal{Y}_0$ as well as the propagation of regularity of $\dv(\omega Y)$. \cref{S:ProofOfSerfatisResult} contains the body of the proof of
\cref{T:FamilyVelOmega2D}. In \cref{S:3DOutline} we outline the changes to the proof of \cref{T:FamilyVelOmega2D} needed to obtain \cref{T:FamilyVelOmega3D} for $d \ge 3$.
\OptionalDetails{In \cref{S:ProofOfVortexPatch} we prove \cref{T:BoundaryRegularity}, which we also shows gives the propogation of regularity of the boundary of a classical 2D vortex patch (\cref{T:VortexPatch}).}

Finally, in \cref{A:TransportEstimates}, we discuss our use of weak transport equations.

\section{Notation, conventions, and definitions}\label{S:Notation}

\noindent

We define $\grad u$, the Jacobian matrix of $u$, as the $d\times d$ matrix with
\begin{align*}
	(\grad u)^i_j
		=\prt_j u^i	
\end{align*} and define the gradient of other vector fields in the same manner. We follow the common convention that 
the gradient and divergence operators apply only to the spatial variables.

We write $C(p_1, \dots, p_n)$ to mean a constant that depends only upon the parameters $p_1, \dots, p_n$. We follow the convention that such constants can vary from expression to expression and even between two occurrences within the same expression.  We will make frequent use of constants of the form,
\begin{align}\label{e:Calpha}
    c_\al
        := C(\omega_0, \Cal{Y}_0) \al^{-1}, \quad
    C_\al
        := C(\omega_0, \Cal{Y}_0) \al^{-1} (1 - \al)^{-1},
\end{align}
where $C(\omega_0, \Cal{Y}_0)$ is a constant that depends upon only $\omega_0$ and $\Cal{Y}_0$.

We define
\begin{align*}
	&M_{m \times n}(\R) = \text{ the space of all $m \times n$ real matrices}, \\
	&M^i_j = \text{the element at the $i$-th row, $j$-th column of } M \in M_{d \times d}(\R), \\
	&M_j = \text{ the $j$-th column of } M \in M_{d \times d}(\R), \\
	&M \cdot N = \sum_{i, j} M^i_j N^i_j = \sum_j M_j \cdot N_j
		\text{ for all } M, N \in M_{m \times n}(\R).
\end{align*}
Repeated indices appearing in upper/lower index pairs are summed over, but no summation occurs if the indices are both upper or both lower.

We write $\abs{v}$ for the Euclidean norm of $v = (v^1, v^2, \cdots, v^{d})$, $\abs{v}^2 = (v^1)^2 + (v^2)^2+\cdots (v^{d})^2$. For $M \in M_{d \times d}(\R)$, we use the operator norm,
\begin{align}\label{e:MatrixNorm}
    \abs{M}
        := \max_{\abs{v} = 1} \abs{M v}.
\end{align}
Of course, all norms on finite-dimensional spaces are equivalent, so the choice of matrix norm just affects the values of constants. Our choice has the convenient properties, however, that it is sub-multiplicative, gives the identity matrix norm $1$, and
\begin{align}\label{e:MatrixNormProperties}
	\abs{M} 
		&= \sqrt{\max \text{ eigenvalue of } M M^*}
		\le \pr{\sum^{d}_{i, j=1} (M^i_j)^2}^{\frac{1}{2}}
		\le \sqrt{d} \abs{M},
\end{align}
the first inequality being strict when $M$ is nonsingular. If $X$ is a function space, we define
\begin{align*}
    \norm{v}_X
        := \norm{\abs{v}}_X, \quad
    \norm{M}_X
        := \norm{\abs{M}}_X.
\end{align*}

\begin{definition}[\Holder and Lipschitz spaces]\label{D:HolderSpaces}
        Let $\alpha \in(0, 1)$ and $U \subseteq \R^d$, $d \ge 1$, be open.
        Then $C^\al(U)$ is the space of all measurable functions for which
        \begin{align*}
            \norm{f}_{C^\alpha(U)}
                := \norm{f}_{L^\iny(U)} + \norm{f}_{\dot{C}^\alpha(U)}
                    < \iny, \quad
            \norm{f}_{\dot{C}^\alpha(U)}
                := \sup_{\substack{x, y \in U \\ x \ne y}}
                    \frac{\abs{f(x) - f(y)}}{\abs{x - y}^\alpha}.
        \end{align*}

        For $\al = 1$, we obtain the Lipschitz space, which is not called $C^1$
        but rather $Lip(U)$. We also define $lip(U)$ for the homogeneous space.
        Explicitly, then,
        \begin{align*}
            \norm{f}_{Lip(U)}
                := \norm{f}_{L^\iny(U)} + \norm{f}_{lip(U)}, \quad
            \norm{f}_{lip(U)}
                := \sup_{\substack{x, y \in U \\ x \ne y}}
                    \frac{\abs{f(x) - f(y)}}{\abs{x - y}}.
        \end{align*}

        For any positive integer $k$, $C^{k + \al}(U)$ is the space
        of $k$-times continuously differentiable functions on $U$ for which
        \begin{align*}
            \norm{f}_{C^{k + \al}(U)}
                := \sum_{\abs{\beta} \le k} \smallnorm{D^\beta f}_{L^\iny(U)}
                    + \sum_{\abs{\beta} = k} \smallnorm{D^\beta f}_{C^\al(U)}
                    < \iny.
        \end{align*}
        We define the negative \Holder space, $C^{\al - 1}(U)$, by
    \begin{align*}
        C^{\al - 1}(U)
            &= \set{f + \dv v \colon f, v \in C^\al(U)}, \\
        \norm{h}_{C^{\al - 1}(U)}
            &= \inf\set{
                \norm{f}_{C^\al(U)}
                    + \norm{v}_{C^\al(U)}
                    \colon
                        h = f + \dv v; \, f, \, v \in C^\al(U)
                }.
    \end{align*}
\end{definition}

It follows immediately from the definition of $C^{\al - 1}$ that
\begin{align}\label{e:dvvCalBound}
    \norm{\dv v}_{C^{\al - 1}}
        \le \norm{v}_{C^\al}.
\end{align}
We also have the elementary inequalities,
\begin{align}\label{e:CalphaFacts}
    \begin{split}
        \norm{f \circ g}_{\dot{C}^\al}
            &\le \norm{f}_{\dot{C}^\al} \norm{\grad g}_{L^\iny}^\al, \\
        \norm{f g}_{C^\al}
            &\le \norm{f}_{C^\al} \norm{g}_{C^\al}, \\
        \norm{1/f}_{\dot{C}^\al}
        	&\le \frac{\norm{f}_{\dot{C}^\al}}{(\inf \abs{f})^2}.
    \end{split}
\end{align}

\begin{definition}[Radial cutoff functions]\label{D:Radial}
    We make an arbitrary, but fixed, choice of a radially symmetric function
    $a \in C_C^\iny(\R^d)$ taking values in $[0, 1]$ with $a = 1$ on $B_1(0)$
    and $a = 0$ on
    $B_2(0)^C$. For $r > 0$, we define the rescaled cutoff function, $a_r(x) = a(x/r)$,
    and for $r, h > 0$ we define
    \begin{align*}
        \mu_{rh} = a_r(1 - a_h).
    \end{align*}
\end{definition}

\begin{remark}\label{R:Radial}
When using the cutoff function $\mu_{rh}$ we will be fixing $r$ while taking $h \rightarrow  0$, in which case we can safely assume that $h$ is sufficiently smaller than $r$ so that $\mu_{rh}$ vanishes outside of $(h, 2r)$ and equals 1 identically on $(2h, r)$. It will then follow that
\begin{align*} 
    \left\{
    \begin{array}{rl}
        \halfspacer
        \abs{\grad \mu_{rh}(x)} \le C h^{-1} \le C \abs{x}^{-1}
            & \text{for } \abs{x} \in (h, 2h), \\
        \halfspacer
        \abs{\grad \mu_{rh}(x)} \le C r^{-1} \le C \abs{x}^{-1}
            & \text{for } \abs{x} \in (r, 2r), \\
        \grad \mu_{rh} \equiv 0
            & \text{elsewhere}.
    \end{array}
    \right.
\end{align*}
Hence, also, $\abs{\grad \mu_{rh}(x)} \le C \abs{x}^{-1}$ everywhere.
\end{remark}

\begin{definition}[Mollifier]\label{D:Mollifier}
    Let $\rho \in C_C^\iny(\R^d)$ with $\rho \ge 0$ have $\norm{\rho}_{L^1} = 1$
    and be radially symmetric.
    For $\eps > 0$, define $\rho_\eps(\cdot) = (\eps^{-d}) \rho(\cdot / \eps)$.
\end{definition}

\begin{definition}[Principal value integral]\label{D:PV}
For any measurable integral \textit{kernel}, $L \colon \R^d\times \R^d \to \R$, and any measurable function, $f \colon \R^d \to \R$, define the integral transform $L[f]$ by
\begin{align*}
    L[f](x)
        := \PV \int_{\R^d} L(x, y) f(y) \, dy
        &:= \lim_{h \to 0^+} \int_{\abs{x - y} > h} L(x, y) \, f(y) \, dy,
\end{align*}
whenever the limit exists.
\end{definition}

Finally, we give the form of Gronwall's lemma that we will need.

\begin{lemma}[Gronwall's lemma and reverse Gronwall's lemma]\label{L:Gronwall}
    Suppose $h \ge 0$ is a continuous nondecreasing or nonincreasing function on $[0,T]$,
    $g \ge 0$ is an
    integrable function on $[0,T]$, and
    \begin{align*}
        f(t)
            \le h(t) + \int_0^t g(s) f(s) \, ds
                \ \text{ or } \ 
        f(t)
            \ge h(t) - \int_0^t g(s) f(s) \, ds     
    \end{align*}
    for all $t \in [0, T]$. Then
    \begin{align*}
        f(t) \le h(t) \exp \int_0^t g(s) \, ds
            \ \text{ or } \ 
        f(t) \ge h(t) \exp \pr{-\int_0^t g(s) \, ds},
    \end{align*}
    respectively, for all $t \in [0, T]$.
\end{lemma}

%
%
\section{Estimates on singular integrals}\label{S:SingularIntegrals}

\noindent Because $\grad u$, via the Biot-Savart law \cref{e:BSLaw}, involves a singular integral, estimates on such integrals are central to all of our results. In this section, we give the basic estimates we will need for such integrals.

\cref{L:SerfatiLemma2} is a fairly standard result on singular integral operators (so we suppress its proof). We do not apply it directly, but rather indirectly through its corollary, \cref{L:SerfatiLemma2Inf}. \cref{L:SerfatiKernels} gives explicit estimates on the kernels to which we apply \cref{L:SerfatiLemma2Inf}. We note that one of these kernels is not derived from the Biot-Savart kernel.

\begin{lemma}\label{L:SerfatiLemma2}
    Let $L \colon \R^d\times \R^d \to \R$ be an integral kernel for which
    \begin{align*}
        \norm{L}_*
            := \sup_{x, y \in \R^d}
            \bigset{\abs{x - y}^d \abs{L(x, y)}
                + \abs{x - y}^{d+1} \abs{\grad_x L(x, y)}}
        < \iny   
    \end{align*}
    and for which
    \begin{align}\label{e:LPVinL1}
        \abs{\PV \int_{\R^d} L(x, y) \, dy}
            < \iny
            \text{ for all } x \in \R^d.
    \end{align}
    Let $L[f]$ be as in \cref{D:PV}.
    Then
    \begin{align}\label{e:KernelEstimate}
        \begin{split}
        \norm{\PV \int_{\R^d} L(x, y)
                \brac{f(y) - f(x)} \, dy}_{\dot{C}_x^\alpha}
            \le C \al^{-1} (1 - \al)^{-1}
                \norm{L}_* \norm{f}_{\dot{C}^\alpha}.
        \end{split}
    \end{align}
    If 
    \begin{align}\label{e:LHomo}
        \PV \int_{\R^d} L(\cdot, y) \, dy \equiv 0
    \end{align}
    then
    \begin{align}\label{e:KernelEstimateHomo}
        \norm{L[f]}_{\dot{C}^\alpha}
            \le C \al^{-1} (1 - \al)^{-1}
                \norm{L}_* \norm{f}_{\dot{C}^\alpha}.
    \end{align}
\end{lemma}

The inequality in \cref{e:KernelEstimateHomo} is a classical result relating a Dini modulus of continuity of $f$ to a singular integral operator applied to $f$ in the special case where the modulus of continuity is $r \mapsto C r^\al$. (See, for instance, the lemma in \cite{KNV2007}, and note that applying that lemma to a $C^\al$ function gives the same factor of $\al^{-1}(1 - \al)^{-1}$ that appears in \cref{L:SerfatiLemma2}. This reflects the fact that the integral transform in \cref{e:KernelEstimate} applied to a $C^1$-function gives only a log-Lipschitz function, and applied to a $C^0$-function yields no modulus of continuity.)

\cref{L:SerfatiLemma2Inf} allows us to bound the full $C^\al$ norm.

\begin{lemma}\label{L:SerfatiLemma2Inf}
    Let $L$ be as in \cref{L:SerfatiLemma2} and suppose further that
    \begin{align*}
        \norm{L}_{**}
            := \norm{L}_*
                + \sup_{x \in \R^d} \norm{L(x, \cdot)}_{L^1(B_1(x)^C)}
            < \iny.
    \end{align*}
    Then the conclusions of \cref{L:SerfatiLemma2} hold with each $\dot{C}^\al$
    replaced by $C^\al$ and $\norm{L}_*$ replaced by $\norm{L}_{**}$.
    \end{lemma}

\begin{proof}
In light of \cref{L:SerfatiLemma2}, we only need to bound the corresponding $L^\iny$ norms. We have,
\begin{align*}
    &\norm{\PV \int_{\R^d} L(\cdot, z) \brac{f(z)-f(\cdot)} \, dz}_{L^\iny} \\
        &\qquad
        \le \norm{f}_{\dot{C}^\al}
            \norm{\lim_{h \to 0} \int_{B_h(x)^C \cap B_1(x)}
            \abs{L(x, z)} \abs{x - z}^\al \, dz}_{L^\iny_x}
        + 2 \norm{f}_{L^\iny} \sup_{x \in \R^2} \norm{L(x, \cdot)}_{L^1(B^1(x)^C)}
        \\
        &\qquad
        \le \norm{L}_* \norm{f}_{\dot{C}^\al}
            \norm{\lim_{h \to 0} \int_{B_h(x)^C \cap B_1(x)}
            \abs{x - z}^{\al - d} \, dz}_{L^\iny_x}
        + 2 \norm{L}_{**} \norm{f}_{L^\iny} \\
        &\qquad
        \le C \al^{-1} \norm{L}_{**} \norm{f}_{C^\alpha}.
\end{align*}
\end{proof}

We shall apply \cref{L:SerfatiLemma2Inf} to the kernels of \cref{L:SerfatiKernels}. Note that for $L_2$, we are actually applying \cref{L:SerfatiLemma2} to each of its components. Also, for no $\eps > 0$ is $L_1$ singular, but it becomes singular in the limit as $\eps \to 0$.

\begin{lemma}\label{L:SerfatiKernels}
	Assume that $\Omega \in L^1 \cap L^\iny(\R^d)$ and define the kernels,
    \begin{enumerate}
        \item
            $L_1(x, y) = \rho_\eps(x - y) \Omega(y)$;

        \item
            $L_2(x, y) = \Omega(y) \grad K_d(x - y)$.
     \end{enumerate}
    Here, $K_d = \grad \Cal{F}_d$ is the Biot-Savart kernel of \cref{e:BSLaw}
    (in 2D, we can use $K$). Then
    $\norm{L_1}_{**} \le C \norm{\Omega_0}_{L^\iny}$ for $C$
    independent of $\eps$ and     $\norm{L_2}_{**} \le C V(\Omega)$ with
    \begin{align}\label{e:Vomega}
        V(\Omega)
            :=
            \norm{\Omega}_{L^\iny}
                + \norm{\PV \int \Omega(y) \grad K_d(x - y) \, dy}_{L^\iny}.
    \end{align}
\end{lemma}

\begin{proof}
    The bounds on the $*$-norms of $L_1$
    and $L_2$ are easily verified, the
    key points being their $L^1$-bound uniform in $x$,
    the decay of $K_d(x - y)$ and $\grad_x K_d(x - y)$,
    and the scaling of $\rho_\eps(x - y)$ and $\grad_x \rho_\eps(x - y)$ in terms of $\eps$.
    The $\PV$ integral in
    \cref{e:Vomega} comes from the final term in $\norm{L}_{**}$.
\end{proof}

\begin{lemma}\label{L:IntCalBound}
Let $r \in (0, 1]$. For all $f \in \dot{C}^\al(\R^d)$, $g \in L^\iny(\R^d)$, we have
\begin{align}\label{e:IntCalBound1}
    &\abs{\int \grad [\mu_{rh} \grad \Cal{F}_d](x - y)
        (f(x) - f(y)) g(y) \, dy}
        \le C \al^{-1} \norm{f}_{\dot{C}^\al} \norm{g}_{L^\iny} r^\al.
\end{align}
For all $f \in L^\iny(\R^d)$, we have
\begin{align}\label{e:IntCalBound2}
	\begin{split}
    &\abs{\int (\mu_{rh} \grad \Cal{F}_d)(x - y) f(y) \,\ dy}
        = \abs{\int (\mu_{rh} \grad \Cal{F}_d)(x - y) (f(y) - f(x))
                \, dy} \\
        &\qquad\qquad\qquad
        \le C \al^{-1} \norm{f}_{C^{\al - 1}} r^\al.
    \end{split}
\end{align}
\end{lemma}
\begin{proof}
	First observe that the integrals in \cref{e:IntCalBound1,e:IntCalBound2}
	are well-defined because in both cases, $f$ is bounded on any compact subset
	of $\R^d$ and the kernels are locally integrable.
	
    For \cref{e:IntCalBound1}, we have
    $\abs{\grad [\mu_{rh} \grad \Cal{F}_d](x - y)}
    \le C s^{-d}$ by \cref{R:Radial}
    and $\abs{f(x) - f(y)}
    \le \norm{f}_{\dot{C}^\al} s^\al$,
    where $s = \abs{x - y}$. Hence,
    \begin{align*}
		&\abs{\int \grad [\mu_{rh} \grad \Cal{F}_d](x - y) (f(x) - f(y)) g(y) \, dy}
			\le C \norm{f}_{\dot{C}^\al} \norm{g}_{L^\iny}
				\int_h^r s^{-d} s^\al s^{d - 1}\, d s \\
			&\qquad
			\le C \al^{-1} \norm{f}_{\dot{C}^\al} \norm{g}_{L^\iny} r^\al.
	\end{align*}
        
    For \cref{e:IntCalBound2},
    radial symmetry of the kernel gives equality of the two forms of the integrals.
    
    Since $f \in L^\iny(\R^d) \subseteq C^{\al - 1}(\R^d)$, we see from
    \cref{D:HolderSpaces} that there exist
    $f_0, f_1 \in C^\al$ with $f = f_0 + \dv f_1$ such that
    $
        \norm{f_0}_{C^\al}, \norm{f_1}_{C^\al}
            \le 2 \norm{f}_{C^{\al - 1}}.
    $
    (The $2$ could be any value greater than $1$.)
    For $f_0$, we have,
    \begin{align*}
		&\abs{\int (\mu_{rh} \grad \Cal{F}_d)(x - y) (f_0(x) - f_0(y)) \, dy}
			\le C \norm{f_0}_{\dot{C}^\al}
				\int_h^r s^{-(d - 1)} s^\al s^{d - 1}\, d s \\
			&\qquad
			\le C \norm{f_0}_{\dot{C}^\al} r^{\al + 1}
			\le C \norm{f_0}_{C^\al} r^{\al + 1}
			\le C \norm{f}_{C^{\al - 1}} r^{\al + 1}
			\le C \norm{f}_{C^{\al - 1}} r^{\al }.
	\end{align*}
	
	Observe that both $f_1$ and $\dv f_1$ are $C^\al$, since $f, f_0 \in C^\al$.
	Hence, we can integrate by parts and use \cref{e:IntCalBound1} to obtain
    \begin{align*}
		&\abs{\int (\mu_{rh} \grad \Cal{F}_d)(x - y) (\dv f_1(x) - \dv f_1(y)) \, dy} \\
			&\qquad
			= \abs{\int \grad [\mu_{rh} \grad \Cal{F}_d](x - y) (f_1(x) - f_1(y)) \, dy}
			\le C \al^{-1} \norm{f_1}_{\dot{C}^\al} r^\al \\
			&\qquad
			\le C \al^{-1} \norm{f_1}_{C^\al} r^\al
			\le C \al^{-1} \norm{f}_{C^{\al - 1}} r^\al.
    \end{align*}
    Adding the bounds for these two integrals yields \cref{e:IntCalBound2}. 
\end{proof}

%
%
\section{Lemmas involving the velocity gradient}\label{A:graduCalcs}

\noindent In this section we give the lemmas involving $\grad u$ that we will need.

\cref{P:graduExp} is a standard way of expressing $\grad u$; it is, in fact, the decomposition of $\grad u$ into its antisymmetric and symmetric parts. It follows, for instance, from Proposition 2.17 of \cite{MB2002}.

In \cref{P:graduYLikeLemma}, we inject the $C^\al$-vector field $Y$ into the formula given in \cref{P:graduExp}; the expression that results lies at the heart of the proofs of \cref{T:FamilyVelOmega2D,T:FamilyVelOmega3D,T:Equivalence}, via \cref{C:graduCor}, and the proof of \cref{T:A}, via \cref{C:graduCorCor2D}. \cref{P:ConvEq} justifies switching between two ways of calculating principal value integrals. \cref{P:EquivalentConditions,L:KcurlZ,} are used in the proofs of these results; \cref{P:EquivalentConditions} is also used directly in the proof of \cref{T:FamilyVelOmega2D}.
\ifbool{ForSubmission}{
    We leave the proofs of \cref{P:graduYLikeLemma,P:ConvEq} to the reader.
    }
    {
    }

Recall the definitions of $K$ and $K_d$ in \cref{e:BSLaw}. We note that $\grad K_d$ is a symmetric matrix.
\begin{prop}\label{P:graduExp}
    Let $u$ be a divergence-free vector field vanishing at infinity with vorticity
    $\Omega \in L^1 \cap L^\iny(\R^d)$. Then
    \begin{align*}
    	\begin{array}{rl}
			\spacer
			d = 2:
        		&\grad u(x)
            		= \displaystyle
						\frac{\omega(x)}{2} \matrix{0 & -1}{1 & 0}
							+ \PV \int \grad K(x - y) \omega(y) \, dy, \\
			d \ge 2:
        		&\grad u(x)
					= \prt_j u^i(x)
            		= \displaystyle
						\frac{\Omega(x)}{2}
                			+ 
							\PV \int \Omega(y) \grad K_d(x - y) \, dy; \\
        		&(\grad u)^i_j(x)
					= \prt_j u^i(x)
            		= \displaystyle
						\frac{\Omega^i_j(x)}{2}
                			+ \PV \int \prt_i \prt_k \Cal{F}_d(x - y) \Omega^j_k(y) \, dy.
        \end{array}
    \end{align*}
    The first term is the antisymmetric, the second term the symmetric part
    of $\grad u(x)$.
\end{prop}

In \cref{P:graduExp}, the principal value integral is a singular integral operator, which is well-defined as a map from $L^p$ to $L^p$ for any $p \in (1, \iny)$. (See, for instance, Theorem 2 Chapter 2 of \cite{S1970}.)

\begin{prop}\label{P:graduYLikeLemma}
    Let $\Omega \in L^1 \cap L^\iny(\R^d)$ and let $Y$ be a vector field in $C^\al(\R^d)$.
    Then
    \begin{align*} 
    	\begin{array}{rl}
			\spacer
			d = 2:
        		&\displaystyle
					\PV \int \grad K(x - y) Y(y) \, \omega(y) \, dy
            			= - \frac{\omega(x)}{2} \matrix{0 & -1}{1 & 0} Y(x)
                			+ \brac{K * \dv(\omega Y)}(x), \\
			d \ge 2:
        		&\displaystyle
					\brac{\PV \int \Omega(y) \grad K_d(x - y) Y(y) \, dy}^j
            			= \displaystyle
							- \frac{(\Omega(x) Y(x))^j}{2} 
                			+ \brac{K_d^k * \dv(\Omega^j_k Y)}(x).
        \end{array}
    \end{align*}
\end{prop}

\begin{cor}\label{C:graduCor}
    Let $\Omega \in L^1 \cap L^\iny(\R^d)$ and let $Y$ be a vector field in $C^\al(\R^d)$.
    Then
    \begin{align*}
    	\begin{array}{rl}
			\spacer
			d = 2:
        		&\displaystyle
        		Y(x) \cdot \grad u(x)
            		= \PV \int \grad K(x - y) \brac{Y(x) - Y(y)} \omega(y) \, dy
                		+ \brac{K * \dv(\omega Y)}(x), \\
			d \ge 2:
        		&\displaystyle
        		\brac{Y(x) \cdot \grad u(x)}^j
            		= \brac{\PV \int \Omega(y) \grad K_d(x - y) \brac{Y(x) - Y(y)} \, dy}^j \\
						&\qquad\qquad\qquad\qquad\qquad
                		+ \brac{K_d^k * \dv(\Omega^j_k Y)}(x). \\
        \end{array}
    \end{align*}
    Moreover, for $d = 2$,
    \begin{align*}
        \norm{\PV \int \grad K(x - y) \brac{Y(x) - Y(y)} \omega(y) \, dy}_{C^\al}
            \le C V(\omega) \norm{Y}_{C^\al},
    \end{align*}
    $V(\omega)$ being given in \cref{e:Vomega}. The analogous bound holds for $d \ge 3$.
\end{cor}

\begin{proof}
The expression for $Y(x) \cdot \grad u(x)$ follows from comparing the expressions in \cref{P:graduExp,P:graduYLikeLemma}. The $C^\al$-bound follows from applying \cref{L:SerfatiLemma2} and \cref{L:SerfatiLemma2Inf} with the kernel $L_2$ of \cref{L:SerfatiKernels}.
\end{proof}

\begin{lemma}\label{L:KcurlZ}
	Let $Z$ be a vector field in $L^1 \cap L^\iny(\R^2)$. Then
	\begin{align*}
		K * \dv Z = Z^\perp - (K* \curl Z)^\perp.
	\end{align*}
\end{lemma}

\begin{proof}
	A direct calculation shows that as tempered distributions,
	the divergence of each side is zero, while 
	the curl of each side is $\dv Z$. Since each side decays at infinity,
	it follows that the two sides are equal
	(see, for instance, Proposition 1.3.1 of \cite{C1998}).
\end{proof}

\begin{prop}\label{P:EquivalentConditions}
    If $Z \in (L^{1}\cap L^{\iny})(\R^d)$ and $\dv Z \in C^{\alpha-1}(\R^d)$, then $\grad \Cal{F}_d \ast \dv Z \in C^\alpha(\R^d)$ (equivalently, $K * \dv Z \in C^\al(\R^2)$, for $d = 2$).
    Moreover,
    \begin{align}\label{e:EquivalentConditionsBound1}
        \norm{\grad \Cal{F}_d * \dv Z}_{C^\al}
            \le C \pr{\norm{Z}_{L^1 \cap L^\iny} + \norm{\dv Z}_{C^{\al - 1}}}.
    \end{align}
    We also have $\dv Z \in C^{\alpha-1}(\R^d)$ if $\grad \Cal{F}_d \ast \dv Z \in C^\alpha(\R^d)$
    (equivalently, $K * \dv Z \in C^\al(\R^2)$, for $d = 2$) and 
    \begin{align}\label{e:EquivalentConditionsBound2}
        \norm{\dv Z}_{C^{\al - 1}}
            \le \norm{\grad \Cal{F}_d * \dv Z}_{C^\al}.
    \end{align}
\end{prop}

\begin{proof}
Suppose that $Z \in (L^{1}\cap L^{\iny})(\R^d)$ with $\dv Z\in C^{\alpha-1}(\R^d)$. We have,
\begin{align*}
    \grad \Cal{F}_{d} * \dv Z
        = m(D) \dv Z
        = n_i(D) Z^i,
\end{align*}
where $m$ and $n_i$, $i = 1, 2, \cdots, d$, are the Fourier-multipliers,
\begin{align*}
    m(\xi) = \frac{\xi}{\abs{\xi}^2}, \quad
    n_{i}(\xi) = \frac{\xi^i \xi}{\abs{\xi}^2},
\end{align*}
up to unimportant multiplicative constants.
We can thus write $\grad \Cal{F}_{d} * \dv Z$ using a Littlewood-Paley decomposition in the form,
\begin{align}\label{e:KdivZDef}
    \grad \Cal{F}_{d} * \dv Z
        = \sum_{j \ge -1} \Delta_j m(D) \dv Z
        = \Delta_{-1} n_i(D) Z^i
            + \sum_{j \ge 0} \Delta_j m(D) \dv Z,
\end{align}
where $\Delta_j$ are the nonhomogeneous Littlewood-Paley operators (dyadic blocks). We use the notation of \cite{BahouriCheminDanchin2011} and refer the reader to Section 2.2 of that text for more details. The sum in \cref{e:KdivZDef} will converge in the space $\Cal{S}'(\R^d)$ of Schwartz-class distributions as long as $\dv Z \in \Cal{S}'(\R^d)$.

Now, for any noninteger $r \in [-1, \iny)$,
\begin{align*}
    \sup_{j \ge -1} 2^{jr} \norm{\Delta_j f}_{L^\iny}
\end{align*}
is equivalent to the $C^r$ norm of $f$ (see Propositions 6.3 and 6.4 in Chapter II of \cite{Chemin2004Handbook}, which apply to all $d \ge 2$). Also,
\begin{align*}
    \norm{\Delta_{-1} n_i(D) f}_{L^\iny}
        \le C \norm{f}_{L^2} \le C \norm{f}_{L^{1}\cap L^{\infty}}, \quad
    \norm{\Delta_j m(D) f}_{L^\iny}
        \le C 2^{-j} \norm{\Delta_j f}_{L^\iny}
\end{align*}
for $j \ge 0$ and $i = 1, 2$. These inequalities follow from Lemma 2.1 and Lemma 2.2 of \cite{BahouriCheminDanchin2011}. Hence,
\begin{align*}
    &\norm{\grad \Cal{F}_d * \dv Z}_{C^{\alpha}}
        \le \left\|\Delta_{-1} n_i(D) Z^i\right\|_{L^{\iny}}
            +\sup_{j\ge 0} 2^{j\al}\left\| \Delta_j m(D)\dv Z \right\|_{L^{\iny}} \\
        &\qquad
        \leq C \|Z\|_{L^2} + \sup_{j\ge 0} 2^{j(\al -1)}\left\| \Delta_j \dv Z \right\|_{L^{\iny}}
        \leq C \|Z\|_{L^{1}\cap L^{\infty}}
            + C \left\| \dv Z\right\|_{C^{\alpha-1}},
\end{align*}
which gives the inequality in \cref{e:EquivalentConditionsBound1}.

Conversely, assume that $v :=\grad \Cal{F}_{d} \ast \dv Z \in C^\alpha(\R^d)$. Then, 
\begin{align*} 
 \begin{split}
   & \dv v=\Delta \Cal{F}_{d} \ast \dv Z=\dv Z.
 \end{split}
\end{align*}
Therefore, we conclude that $\dv Z \in C^{\alpha - 1}(\R^d)$ and obtain the inequality in \cref{e:EquivalentConditionsBound2}.
\end{proof}

\begin{cor}\label{C:graduCorCor2D}[2D]
    Let $\omega \in L^1 \cap L^\iny(\R^2)$ and let $Y$ be a vector field in $C^\al(\R^2)$
    with $\dv (\omega Y) \in C^{\al - 1}$.
    Then
    \begin{align*}
        Y^\perp(x) \cdot \grad u(x)
            &= \PV \int \grad K(x - y) \brac{Y^\perp(x) - Y^\perp(y)} \omega(y) \, dy
                		+ \brac{K * \dv(\omega Y)}^\perp(x)
						- \omega Y(x).
    \end{align*}
    Moreover,
    \begin{align*}
        \smallnorm{Y^\perp \cdot \grad u + \omega Y}_{C^\al}
            \le C V(\omega) \norm{Y}_{C^\al}
            	+ C \smallnorm{\dv (\omega Y)}_{C^{\al - 1}}.
    \end{align*}
\end{cor}

\begin{proof}
	Applying \cref{L:KcurlZ} with $Z = \omega Y^\perp$ gives
	\begin{align*}
		K * \dv(\omega Y^\perp)
			= (\omega Y^\perp)^\perp - (K * \curl(\omega Y^\perp))^\perp
			= - \omega Y + (K * \dv(\omega Y))^\perp.
	\end{align*}
	Applying \cref{C:graduCor} with $Y^\perp$ in place of $Y$ then gives the expression for
	$Y^\perp \cdot \grad u$, and the $C^\al$ bound on $Y^\perp \cdot \grad u + \omega Y$ follows
	as in the proof of \cref{C:graduCor}, and using \cref{P:EquivalentConditions}.
\end{proof}

\begin{prop}\label{P:ConvEq}
    Let $f \in C^\beta(\R^d)$ for $\beta > 0$ be. Then for all $r > 0$.
    \begin{align*}
    	\PV \int (a_r K)(x - y) f(y) \, dy
            = \lim_{h \to 0} \grad (\mu_{rh} K) * f(x),
    \end{align*}
    where $a_{r}$ and $\mu_{rh}$ are defined in \cref{D:Radial}.
\end{prop}

%
%
\section{Serfati's linear algebra lemma}\label{S:Lemmas}

\noindent In this section we state and prove a simple linear algebra lemma due to Serfati. (The authors are not aware of earlier versions of similar lemmas. This lemma is key in Serfati's approach; it does not appear, for instance, in Chemin's approach in \cite{Chemin1991VortexPatch,Chemin1993Persistance}.)
This lemma will be used both in the establishing the equivalence of striated vorticity and velocity in \cref{S:Equivalence} and in proving the propagation of striated vorticity in \cref{S:ProofOfSerfatisResult,S:3DOutline}.

The $2D$ version of \cref{L:SerfatiLemma1} appeared, in slightly different form, in \cite{SerfatiVortexPatch1994}. A  version for $d \ge 2$ appeared in Serfati's doctoral thesis, \cite{SerfatiThesis}, and in \cite{Serfati3DStratified}. The proof we give is an elucidation of the short proof that appears in \cite{SerfatiUnpublished3DVortexPatches}.

In \cref{L:SerfatiLemma1}, we use the space $\widetilde{M}_{d \times d}(\R)$ of all matrices in $M_{d \times d}(\R)$ with the special property that the last column of each matrix in $\widetilde{M}_{d \times d}(\R)$ is the same as the last column of its cofactor matrix. This means that the first $d - 1$ columns of $M$ uniquely determine the last column. Hence, the polynomials in \cref{L:SerfatiLemma1} can be treated as functions of the first $d - 1$ columns---it is in this sense that we state the degrees of the polynomials.

Observe that in \cref{L:SerfatiLemma1} the final column of $B M$ does not appear in the bound on $\abs{B}$. The reason this will be useful is that in our application of it, the first $d - 1$ columns of $M$ will represent the $d - 1$ directions in which we have regularity of the velocity. This will give us control of $B M_i$ for $i < d$. Then the final column of $M$ will be the wedge product of the other columns, so that $M$ will lie in the $\widetilde{M}_{d \times d}(\R)$ space of \cref{L:SerfatiLemma1}.

In 2D, the restriction $M \in \widetilde{M}_{2 \times 2}(\R)$ is not required, and we can be more explicit about its bound. We use only the general-dimension estimate, however, so we do not include the 2D proof.
\ifbool{ForSubmission}{(See, however, \cref{R:SerfatiLemma2D}.)}{}

\begin{lemma}\label{L:SerfatiLemma1}
\noindent
	Let $d \ge 2$.
	There exist polynomials, $P_1, P_2 \colon
	\widetilde{M}_{d \times d}(\R) \to [0, \iny)$, such that if
	$B \in M_{d \times d}(\R)$ is symmetric and $M \in \widetilde{M}_{d \times d}(\R)$
	is invertible then
	\begin{align*}
		\abs{B}
			&\le \frac{P_1(M)}{\abs{\det M}^2}
				\sum_{i=1}^{d - 1} \abs{B M_i} + \frac{P_2(M)}{\abs{\det M}} \tr B.
	\end{align*}
	The polynomial $P_1(M)$ is homogeneous of degree $4d - 3$, while $P_2(M)$ is homogeneous of degree
	$2d - 2$, in $M_1, \dots, M_{d - 1}$.
	
\smallskip

	For any symmetric $B \in M_{2 \times 2}(\R)$ and $M \in M_{2 \times 2}(\R)$ with $M$ invertible,
	\begin{align*}
		\abs{B}
        	&\le 2 \frac{\abs{M}^3}{\abs{\det M}} \abs{B M_1}
            	+ \abs{\tr B}.
	\end{align*}
\end{lemma}

\begin{proof}
First, we make the following two simple observations applying to any $D, E \in M_{d \times d}(\R)$:
\begin{align*}
	&(D^T E)^i_j
				= D_i^T E_j
				= D_i \cdot E_j, \\
	&			D E_i = (DE)_i.
\end{align*}
We will use these observations below without comment. Define
\begin{align*}
	M'
		:=
		\begin{pmatrix}
			M_1 & M_2 & \cdots & M_{d - 1} & \ul{M}_d
		\end{pmatrix}.
\end{align*}
Let $\ul{M}$ be the cofactor matrix of $M$ and $\ul{M}'$ the cofactor matrix of $M'$.
Then
\begin{align*}
	\ul{M}' (M')^T
		= \det M' \, I, \quad
	M \ul{M}^T
		= \det M \, I,
\end{align*}
from which it follows that
\begin{align}\label{e:B1}
	B
		= \frac{\ul{M}'}{\det (M M')} D \ul{M}^T, \quad
	D := (M')^T B M.
\end{align}
Thus,
\begin{align*}
	\abs{B}
		\le \frac{\abs{\ul{M}'}}{\det M \det M'} \abs{\ul{M}} \abs{D}.
\end{align*}
We will show that $\abs{D}$ can be bounded in a manner that does not involve the column $B M_d$.

Now,
\begin{align*}
	D^i_j
		&
		= M'_i \cdot B M_j
		= \left\{
			\begin{array}{ll}
				M_i \cdot B M_j, & i < d, \\
				\ul{M}_d \cdot B M_j, & i = d.
			\end{array}
		\right.
\end{align*}

If $j < d$ then $B M_d$ does not appear in $D^i_j$; hence, there are two cases to deal with, $i = d$ and $i < d$.

We deal first with $i = d$ (and $j = d$). We have,
\begin{align*}
	\sum_{i = 1}^d \ul{M}_i \cdot B M_i
		= \sum_{i = 1}^d  (\ul{M}^T B M)^i_i
		= \tr (\ul{M}^T B M)
		= \tr (M\ul{M}^T  B)
		= \det M \tr B,
\end{align*}
since $\tr(DE) = \tr(ED)$ for any $D$, $E$ in $M_{d \times d}(\R)$. This implies that
\begin{align}\label{e:S5}
	\ul{M}_d \cdot (BM)_d
		= \det M \, \tr B - \sum_{i = 1}^{d - 1} \ul{M}_i \cdot B M_i.
\end{align}
This yields a bound on $D^d_d$ in which the column $B M_d$ never appears.

Now assume that $i < d$ (and $J = d$). 
We have,
\begin{align*}
	M_i \cdot B M_d
		&= (M^T B M)^i_d
		= ((M^T B M)^T)^d_i
		= (M^T B M)^d_i
		= M_d \cdot B M_i.
\end{align*}
Here, we used the symmetry of $B$ for the first and only time. Since $i < d$, we have bounded the remaining components of $D$ without involving the column $B M_d$. We can see, then, that
\begin{align*}
	\abs{B}
		&\le \frac{P_1(M)}{\abs{\det M} \abs{\det M'}}
			\sum_{i=1}^{d - 1} \abs{B M_i} + \frac{P_2(M)}{\abs{\det M'}} \tr B.
\end{align*}
But $M = M'$, since we assumed that $M \in \widetilde{M}_{d \times d}(\R)$, and the result follows.
\end{proof}

\ifbool{ForSubmission}
	{
	\begin{remark}\label{R:SerfatiLemma2D}
		When $d = 2$, we have $\det M' = \abs{M_1}^2$. Then by Hadamard's inequality
		(\cite{Hadamard1893}), 
		$\abs{M_1}^{-1} \le \abs{M_2} \abs{\det M}^{-1} \le \sqrt{2} \abs{M} \abs{\det M}^{-1}$,
		which ultimately leads to the 2D bound on $\abs{B}$.
	\end{remark}
    }
    {
    }

\section{
Equivalence of striated vorticity and velocity}\label{S:Equivalence} 

\noindent

To prove \cref{T:Equivalence}, we
first show that $\grad u \in L^\iny$. This can be done via a direct calculation, simple in 2D, but substantially more involved in higher dimensions. The idea behind this bound is that $\grad u$ is bounded by assumption  in the $d - 1$ directions determined at any point by elements of $\Cal{Y}$, while the divergence-free condition on $u$ along with the boundedness of $\Omega$ are sufficient to control $\grad u$ in $L^\iny$ in the remaining direction.

The proof we give, however, will rely instead on \cref{L:SerfatiLemma1}. This will allow us to obtain the bound on $\grad u \in L^\iny$ very easily in a manner that works for all dimensions 2 and higher. (We will use \cref{L:SerfatiLemma1} again in the proofs of \cref{T:FamilyVelOmega2D,T:FamilyVelOmega3D}.)

\begin{remark}
	Observe that $Y \cdot \grad u = \grad u Y$. We write $Y \cdot \grad u$
	when we wish to emphasize the role of $Y \cdot \grad$
	as a directional derivative (as we do in all sections but this one).
	We write $\grad u Y$ when primarily performing linear algebra manipulations.
\end{remark}

\begin{prop}\label{P:graduLInf}
	Assume that $\Cal{Y} \cdot \grad u \in L^\iny(\R^d)$, $\Omega \in L^\iny(\R^d)$, and $\Cal{Y} \in L^\iny(\R^d)$.
	Then $\grad u \in L^\iny$.
\end{prop}
\begin{proof}
	Fix $x \in \R^d$ and let $Y_1, \dots, Y_{d - 1} \in \Cal{Y}$ have lengths
	of at least $I(\Cal{Y})$ and be such that
	$\abs{\wedge_{i < d} Y_i} \ge I(\Cal{Y})$ as well. This is always
	possible by the definition of $I(\Cal{Y})$.

	Now,
	\begin{align*}
		\abs{\grad u(x)}
			\le \frac{1}{2} \abs{B} + \frac{1}{2} \norm{\Omega(u)}_{L^\iny},
	\end{align*}
	where
	$B = \grad u(x) + (\grad u(x))^T$.
	Since $B$ is symmetric, we can apply \cref{L:SerfatiLemma1} to bound it.
	
	Define $M \in \widetilde{M}_{d \times d}(\R)$ by
	\begin{align*}
		M
			&=
			\begin{pmatrix}
				Y_1 & \cdots Y_{d - 1} & \wedge_{i < d} Y_i
			\end{pmatrix},
	\end{align*}
	so that
	\begin{align*}
		\det M
			= \abs{\wedge_{i < d} Y_i}^2
			\ge I(\Cal{Y})^2.
	\end{align*}
	Then, since $\tr B= 2 \dv u= 0$, \cref{L:SerfatiLemma1} gives
	\begin{align*}
		\abs{B}
			&\le \frac{P_1(M)}{I(\Cal{Y})^2}
				\sum_{i=1}^{d - 1} \abs{B Y_i}
			\le C \frac{\norm{\Cal{Y}}_{L^\iny(\R^d)}^{4d - 3}}{I(\Cal{Y})^2}
				\sum_{i=1}^{d - 1} \abs{B Y_i}
			\le C(\Cal{Y}) \sum_{i=1}^{d - 1} \abs{B Y_i}.
	\end{align*}
	
	But, writing $B = 2 \grad u - (\grad u - (\grad u)^T) = 2 \grad u - \Omega(u)$, we see that
	\begin{align*}
		\abs{B Y_i}
			\le 2 \norm{\Cal{Y} \cdot \grad u}_{L^\iny(\R^d)}
				+ \norm{\Omega(u)}_{L^\iny(\R^d)} \norm{\Cal{Y}}_{L^\iny(\R^d)},
	\end{align*}
	which completes the proof.
\end{proof}

The forward implications in \cref{T:Equivalence} follow from Lemma 4.6 of \cite{Fanelli2012} combined with \cref{P:graduLInf}, which in turn relies upon a key paraproduct estimate of Danchin's in Section 2 of \cite{Danchin1999} . We give a self-contained, elementary proof below that uses, however, the additional assumption in $d \ge 3$ that $\grad \Cal{Y} \in L^\iny(\R^d)$. (Because we apply \cref{T:Equivalence} only at the initial time, the restriction that $\grad \Cal{Y} \in L^\iny$ would need only be imposed on the initial data---the pushforward of the sufficient family need not be Lipschitz, nor should we expect it to be.)

\begin{proof}[\textbf{Proof of \cref{T:Equivalence}}]
	That
	$\dv (\omega \Cal{Y}) \in C^{\al - 1}
    \implies \Cal{Y} \cdot \grad u \in C^\al$ in 2D and that
    $\dv (\Omega^j_k \Cal{Y}) \in C^{\al - 1} \, \forall \, j, k
        	\implies \Cal{Y} \cdot \grad u \in C^\al$ in higher dimensions follow
	by applying \cref{R:CheminDanchinVelocity} at $t = 0$.
	It remains to prove the forward implications in
	\cref{e:RegEquivalence}.
	
	So assume that $\Cal{Y} \cdot \grad u \in C^\al$, imposing the additional assumption
	that $\grad \Cal{Y} \in L^\iny(\R^d)$, as explained above.

	If $d = 2$
    then $\grad u \in L^\iny$
    by \cref{P:graduLInf}, and $\dv (\omega \Cal{Y}) \in C^{\al - 1}$
    follows immediately from \cref{C:graduCor,P:EquivalentConditions}.
    
	Now assume that $d \ge 3$ and that $\Cal{Y}$ is Lipschitz. 
	We have, for any $i$, $k$,
	\begin{align*}
		\prt_k (Y \cdot \grad u)^i - \prt_i (Y \cdot \grad u)^k
			\in C^{\al - 1}
	\end{align*}
	by \cref{e:SimpleCalphaLemma}, below. But,
	\begin{align*}
		\prt_k (Y &\cdot \grad u)^i - \prt_i (Y \cdot \grad u)^k
			= \prt_k (Y^j \prt_j u^i) - \prt_i (Y^j \prt_j u^k) \\
			&= Y^j \prt_j (\prt_k u^i - \prt_i u^k)
				+ \prt_k Y^j \prt_j u^i - \prt_i Y^j \prt_j u^k \\
			&= Y \cdot \grad \Omega_k^i
				+ \brac{\grad Y (\grad u)^T - \grad u (\grad Y)^T}^i_k.
	\end{align*}
	
	Fix $p \in (1, \iny)$. Then
	$\grad u \in L^p \cap L^\iny$ for any $p \in (1, 2)$, because of
	\cref{P:graduLInf} and because
	$\norm{\grad u}_{L^p} \le C(p) \norm{\Omega}_{L^p}$ (a form of the
	Calderon-Zygmun inequality). Hence, $\brac{\grad Y (\grad u)^T - \grad u (\grad Y)^T}
	\in L^p \cap L^\iny \subseteq C^{\al - 1}$,
	so that then
	$Y \cdot \grad \Omega_k^i \in C^{\al - 1}$ for all $i, k$.
	But,
	\begin{align*}
		Y \cdot \grad \Omega_k^i
			= \dv (Y \Omega_k^i) - (\dv Y) \Omega_k^i
	\end{align*}
	and $\dv (Y) \Omega_k^i \in L^1 \cap L^\iny \subseteq C^{\al - 1}$.
	Hence, $\dv (Y \Omega_k^i) \in C^{\al - 1}$.
\end{proof}

\begin{remark}
	It follows from the proof of \cref{T:Equivalence}, Lemma 4.6 of \cite{Fanelli2012},
	and \cref{P:graduLInf} that if $\Cal{Y} \cdot \grad u \in C^\al$ then it must be that
	$\grad Y (\grad u)^T - \grad u (\grad Y)^T \in C^{\al - 1}$.
\end{remark}

We used the following simple lemma above:

\begin{lemma}\label{e:SimpleCalphaLemma}
	If $f \in C^\al$ then $\prt_j f \in C^{\al - 1}$ with
	$\norm{\prt_j f}_{C^{\al - 1}} \le \norm{f}_{C^\al}$.
\end{lemma}
\begin{proof}
	We have $\prt_j f = \dv (f \bm{e}_j)$, where $f \bm{e}_j \in C^\al$.
\end{proof}

\section{
Higher regularity of corrected velocity gradient in 2D} 
\label{S:MatrixA}

\noindent To obtain \cref{T:A}, we need to construct a partition of unity associated to the sufficient family of $C^\al$ vector fields, $\Cal{Y}$, as in the following proposition:

\begin{prop}\label{P:POU}
	Let $\Cal{Y} = (Y^{(\la)})_{\la \in \Lambda}$ be a sufficient family of $C^\al$ vector fields.
	There exists an $R > 0$, $M_0 = C(\Cal{Y}, \al) > 0$,
	and a partition of unity,
	$(\varphi_n)_{n \in \N}$, with the property that for all $n \in \N$,
	\begin{align}\label{e:POUProps}
	\begin{split}
		&\norm{\varphi_n}_{C^\al} \le M_0, \\
		&\exists \, Y \in \Cal{Y} \text{ such that }
			\abs{Y} > I(\Cal{Y})/2 \text{ on } \supp \varphi_n, \\
		&\# \set{k \in \N \colon \supp \varphi_n \cap \supp \varphi_k \ne \emptyset}
			\le 2.
	\end{split}
	\end{align}
\end{prop}

\begin{proof}
	Because $\Cal{Y}$ is $C^\al$, there is a modulus of continuity 	that applies uniformly
	to all elements of $\Cal{Y}$. It follows that there exists some $R > 0$ such that
	for any $x \in \R^2$ there exists some $Y \in \Cal{Y}$ such that
	$\abs{Y} > I(\Cal{Y})/2$ on $B_R(x)$.
	
	Now let $f \in C_0^\iny((0, 1))$ taking values in $[0, 1]$ with $f \equiv 1$
	on $(1/2, 3/4)$.
	Then extend $f$ to be periodic on all of $\R$.
	For any $i, j \in \Z$ define $f_{i j}, g_{i j} \in C_0^\iny(\R^2)$ by
	\begin{align*}
		&f_{i j}(x_1, x_2)
			= f(x_1) f(x_2)
				\text{ on } [i, i + 1] \times [j, j + 1], \\
		&g_{i j}(x_1, x_2)
			= 1 - f(x_1) f(x_2)
				\text{ on } [i + \frac{1}{2}, i + \frac{3}{2}]
					\times [j + \frac{1}{2}, j + \frac{3}{2}], \\
		&f_{i j}, g_{i j} = 0 \text{ elsewhere in } \R^2.
	\end{align*}
	Let $(\varphi_n)_{n \in \N}$ consist of the collection of all the $f_{i j}(\cdot/R)$
	and $g_{i j}(\cdot/R)$ functions indexed in an arbitrary manner.
	It is easy to see that all the properties in \cref{e:POUProps} hold.
\end{proof}

From \cref{P:POU}, with \cref{e:gradetaMainBound,e:CalphaFacts}, \cref{L:ProductCalBoundSum,L:ProductCalBound} follow easily. (Note that $\cref{e:POUProps}_3$ is critical to obtaining these bounds, though the $2$ could be any finite number.)

\begin{lemma}\label{L:ProductCalBoundSum}
    Let $(f_n)_{n \in \N}$
    be a sequence of functions with $f_n \in C^\al(\supp (\varphi_n \circ \eta^{-1}))$ for all $n$.
    Then
    \begin{align*}
        \norm{\sum_{n \in \N} \varphi_n(\eta^{-1}) f_n}_{C^\al(\R^2)}
            \le C e^{e^{C(\omega_0, \Cal{Y}_0) t}} \sup_{n \in \N}
                \norm{f_n}_{C^\al(\supp (\varphi_n \circ \eta^{-1}))}.
    \end{align*}
\end{lemma}

\begin{lemma}\label{L:ProductCalBound}
    Assume that $\varphi \in C_C^\iny(\R^2)$ takes values in $[0, 1]$
    and let $f \in C^\al(\R^2)$.
    Then
    \begin{align*}
        \norm{\varphi f}_{C^\al(\R^2)}
            \le \norm{f}_{L^\iny(\supp \varphi)}
                + \norm{\varphi}_{\dot{C}^\al} \norm{f}_{C^\al(\supp \varphi)}.
    \end{align*}
\end{lemma}

We now have the machinery we need to prove \cref{T:A} in 2D.

\begin{proof}[\textbf{Proof of \cref{T:A}} in 2D]
For any $n \in \Z$ let $Y_n^0 \in \Cal{Y}_0$ be such that $\abs{Y^0_n} > I(\Cal{Y})/2$ on $\supp \varphi_n$, and let $Y_n$ be the pushforward of $Y_n^0$ under the flow map, $\eta$. Define for all $t \ge 0$,
\begin{align}\label{e:AExplicit}
    A_n
        := \frac{1}{\abs{Y_n}^2}
                \matrix
                    {Y_n^1 Y_n^2 & - (Y_n^1)^2}
                    {(Y_n^2)^2 & -Y_n^1 Y_n^2},
		\qquad
    A
        := \sum_n \varphi_n(\eta^{-1}) A_n,
\end{align}
setting $A_n = 0$ outside of $\supp \varphi_n$.
A simple calculation shows that
\begin{align}\label{e:AY2D}
	A_n Y_n = 0, \quad A_n Y_n^\perp = - Y_n.
\end{align}

Let $V_n = \supp \varphi_n(\eta^{-1})$ and note that $\abs{Y_n(t)} > I(\Cal{Y}(t))/2$ on $V_n$ for all $n$. 
Using \cref{e:CalphaFacts},
\begin{align*}
	\norm{A_n(t)}_{C^\al(V_n)}
		&\le \norm{Y_n(t)}_{C^\al(V_n)}^4/I(\Cal{Y}(t))^2.
\end{align*}
The bound on $\norm{A}_{C^\al}$ in \cref{e:ABound} follows, then, from \cref{L:ProductCalBound,L:ProductCalBoundSum}, \cref{e:MainBoundsY}, and \cref{e:IYBound}.

By \cref{e:AY2D}, $(\grad u  - \omega A_n) Y_n = \grad u Y_n  \in C^\al(V_n)$ with norm bounded uniformly over $n$ by \cref{T:FamilyVel}. Also,
\begin{align*}
	(\grad u - \omega A_n) Y_n^\perp
		= \grad u Y_n^\perp + \omega Y_n
		\in C^\al(V_n)
\end{align*}
with norm bounded uniformly over $n$ by \cref{C:graduCorCor2D,T:FamilyVel}. Since in the (orthogonal)  basis, $\set{Y_n, Y_n^\perp}$, the matrix $\grad u - \omega A_n$ is
\begin{align*}
	\matrix{(\grad u - \omega A_n) Y_n}{(\grad u - \omega A_n) Y_n^\perp}^T,
\end{align*}
and $Y_n \in C^\al$ with $\norm{Y_n}_{C^\al(V_n)}$ uniformly bounded, it follows that $\grad u - \omega A_n \in C^\al(V_n)$ with norm bounded uniformly over $n$. Hence, $\grad u - \omega A \in C^\al$ with the bound in \cref{e:ABound}.
\end{proof}

\section{Higher regularity of corrected velocity gradient in 3D} 
\label{S:MatrixA3D}

\noindent As in \cref{S:MatrixA}, we need a partition of unity, as provided by \cref{P:POU3D}, the 3D analog of \cref{P:POU}.

\begin{prop}\label{P:POU3D}
	Let $\Cal{Y} = (Y^{(\la)})_{\la \in \Lambda}$ be a 3D sufficient family of $C^\al$ vector fields.
	There exists an $R > 0$, $M_0 = C(Y, \al) > 0$,
	and a partition of unity,
	$(\varphi_n)_{n \in \N}$, with the property that for all $n \in \N$,
	\begin{align*}
	\begin{split}
		&\norm{\varphi_n}_{C^\al} \le M_0, \\
		&\exists \, Y_1, Y_2 \in \Cal{Y} \text{ such that }
			\abs{Y_1}, \abs{Y_2}, \abs{Y_1 \times Y_2}
				 > I(\Cal{Y})/2 \text{ on } \supp \varphi_n, \\
		&\# \set{k \in \N \colon \supp \varphi_n \cap \supp \varphi_k \ne \emptyset}
			\le 2.
	\end{split}
	\end{align*}
\end{prop}

\begin{proof}
	A minor variant of that of \cref{P:POU}.
\end{proof}

For the remainder of this section, we give only the local argument, dealing with one pair of vector fields $Y_1$, $Y_2 \in \Cal{Y}$ satisfying $\abs{Y_1}, \abs{Y_2}, \abs{Y_1 \times Y_2} \ge I(\Cal{Y})/2$ on some open set, $U = \supp \varphi_k$. This yields locally a matrix field which we will call, $A$. Piecing these matrices together to form a single matrix field is done just as in \cref{S:MatrixA}, so we suppress the details.

Gram-Schmidt orthonormalization yields a $C^\al$ map, $\Cal{G}$, from $\set{Y_1, Y_2}$ to $\set{Y_1', Y_2'}$ that makes $\set{Y_1', Y_2', Y_1' \times Y_2'}$ an orthonormal frame on $U$ in the standard orientation and is such that $\norm{\Cal{G}}_{C^\al} \le C \norm{\Cal{Y}}_{C^\al}$. We suppress this map and simply relabel $Y_1', Y_2'$ as $Y_1, Y_2$, so that $\set{Y_1, Y_2, Y_1 \times Y_2}$ is an orthonormal frame.

We can decompose $\vec{\omega}$ using our orthonormal frame as
\begin{align}\label{e:omegaDecomp}
	\vec{\omega}
		= a_1 Y_1 + a_2 Y_2 + a_3 Y_1 \times Y_2,
\end{align}
where each $a_j$ is a function of space.

\begin{prop}\label{P:OmegaCalphaOneDirection}
	Writing $\vec{\omega}$ as \cref{e:omegaDecomp}, we have $a_3 \in C^\al$ with
	$\norm{a_3}_{C^\al} \le 2 \norm{Y}_{C^\al} \norm{Y \cdot \grad u}_{C^\al}$.
\end{prop}

\begin{proof}
	Fix a point $x \in U$ and let $Z_1 = Y_1(x)$, $Z_2 = Y_2(x)$.
	(In effect, we are freezing the frame as give at the point, $x$,
	and calculating the third component of the curl in an orthonormal
	frame in the standard orientation.)
	Since $\set{Z_1, Z_2, Z_1 \times Z_2}$ is orthonormal, we have,
	\begin{align*}
		a_3(x)
			&= \prt_{Z_1} (u \cdot Z_2) - \prt_{Z_2} (u \cdot Z_1)
			= Z_1 \cdot \grad (u \cdot Z_2) - Z_2 \cdot \grad (u \cdot Z_1) \\
			&= (Z_1 \cdot \grad u) \cdot Z_2 - (Z_2 \cdot \grad u) \cdot Z_1
			+ (Z_1 \cdot \grad Z_2 - Z_2 \cdot \grad Z_1) \cdot u \\
			&= (Y_1(x) \cdot \grad u) \cdot Y_2(x) - (Y_2(x) \cdot \grad u) \cdot Y_1(x).
	\end{align*}
	The last equality holds because $Z_1$, $Z_2$ are constant throughout space.
	We conclude that
	\begin{align*}
		a_3
			&= (Y_1 \cdot \grad u) \cdot Y_2 - (Y_2 \cdot \grad u) \cdot Y_1.
	\end{align*}
	
	But, $(Y_1 \cdot \grad u) \cdot Y_2 \in C^\al$ since $Y_1 \cdot \grad u \in C^\al$,
	$Y_2 \in C^\al$ by assumption and $C^\al$ is an algebra. Similarly,
	$(Y_2 \cdot \grad u) \cdot Y_1 \in C^\al$. Hence, $a_3 \in C^\al$.
\end{proof}

To determine what form the matrix $A$ might take, let us return for a moment to the 2D result of \cref{S:MatrixA}. There, we found that the irregularities in the velocity gradient could be corrected by subtracting from it a matrix-multiple of the scalar vorticity; that is, $\grad u - \omega A \in C^\al$, where $A \in C^\al$ is given by \cref{e:AExplicit}.
There is no correction in the tangential direction, since $\omega A Y = 0$, and a correction tangential to the boundary in the normal direction.
Also, $\omega A Y^\perp = - \omega Y$, so the discontinuity in $\grad u$ in the normal direction is in the tangential direction.

To extend this result to 3D, it will be more convenient to use (mostly) the vorticity in the form of an antisymmetric matrix as opposed to a three-vector. Toward this end, observe that in 2D, a simple calculation shows that
\begin{align}\label{e:AOmega2D}
    \omega A
        = \sum_n \frac{\varphi_n(\eta^{-1})}{\abs{Y_n}^2}
                \matrix
                    {(Y_n^1)^2 & Y_n^1 Y_n^2}
                    {Y_n^2 Y_n^1 & (Y_n^2)^2}
                    \Omega 
        = \brac{\sum_n \frac{\varphi_n(\eta^{-1})}{\abs{Y_n}^2}
                    Y_n \otimes Y_n}
                    \Omega.
\end{align}
So if we had instead defined $A$ to be equal to the expression in brackets on the right-hand side we would have expressed our result in the form $A \Omega$ rather than $\omega A$, and this form makes sense in any number of dimensions.

The analog of the relations  $\omega A Y = 0$, $\omega A Y^\perp = - \omega Y$ in 3D are that
\begin{align}\label{e:AGoal}
	\begin{split}
          	&A \Omega (Y_1 \times Y_2) = \Omega (Y_1 \times Y_2),\\
		&A P_{\spn\set{Y_1, Y_2}} \Omega Y_1
			= a_{3}Y_{2}, \\
		& A P_{\spn\set{Y_1, Y_2}}  \Omega Y_2 = -a_{3}Y_{1},
	\end{split}
\end{align}
where $P_V$ is projection into the subspace $V$. We derive such a matrix $A$ in \cref{P:A3D}, below, but first we show in \cref{L:LocalACalpha} that \cref{e:AGoal} gives, in fact, the required properties.

To prove \cref{e:AGoal}, we will find it useful to have a way to translate between the three-vector and antisymmetric forms of the vorticity by defining, for any three-vector, $\varphi = \innp{\varphi^1, \varphi^2, \varphi^3}$,
\begin{align*}
    Q(\varphi)
        &= \matrixthree{0 & -\varphi^3 & \varphi^2}
                   {\varphi^3 & 0 & -\varphi^1}
                   {-\varphi^2 & \varphi^1 & 0}.
\end{align*}
Then $Q$ is a bijection from the space of $3$-vectors to the space of antisymmetric $3 \times 3$ matrices.
A direct calculation shows that
\begin{align}\label{e:CrossProductEquiv}
    Q(\varphi) v
        = \varphi \times v
\end{align}
for any three-vectors, $\varphi$, $v$.

If $V \subseteq \R^3$ is a subspace,
we define
\begin{align*}
	P_V \Omega := Q(\proj_V \vec{\omega}).
\end{align*}

\begin{prop}\label{L:LocalACalpha}
	Suppose that $A \in C^\al$ satisfies \cref{e:AGoal}. Then
	\begin{align*}
		\grad u - A \Omega \in C^\al.
	\end{align*}
\end{prop}

\begin{proof}
	Let $V = \spn \set{Y_1, Y_2}$ so that $V^\perp = \spn \set{Y_1 \times Y_2}$. Then, 
	\begin{align*}
		&(\grad u - A \Omega) Y_1
			= (\grad u - A P_V \Omega) Y_1 - A P_{V^\perp} \Omega Y_1
			= \grad u Y_1- a_{3}Y_{2}-A P_{V^\perp} \Omega Y_1,\\
			& (\grad u - A \Omega) Y_2
			= (\grad u - A P_V \Omega) Y_2 - A P_{V^\perp} \Omega Y_2
			= \grad u Y_2 +a_{3}Y_{1}-A P_{V^\perp} \Omega Y_2
	\end{align*}
	by \cref{e:AGoal}.
	But $\grad u Y_j - A P_{V^\perp} \Omega Y_j \in C^\al$
	since $\grad u Y_j, A, Y_j \in C^\al$ by assumption and $a_{3}\in C^\al$ and $P_{V^\perp} \Omega \in C^\al$ by \cref{P:OmegaCalphaOneDirection}.
	
	Also,
	\begin{align*}
		(\grad u - A \Omega) (Y_1 \times Y_2)
			&= (\grad u - \Omega) (Y_1 \times Y_2)
			= (\grad u)^T (Y_1 \times Y_2) \in C^\al
	\end{align*}
	by \cref{L:graduTCalpha}.
	
	Because $(\grad u - A \Omega) Y_1$, $(\grad u - A \Omega)Y_2$, and
	$(\grad u - A \Omega) (Y_1 \times Y_2)$ are $C^\al$ and the 	
	Gram-Schmidt orthonormalization map, $\Cal{G}$, is $C^\al$, it follows that
	$\grad u - A \Omega \in C^\al$.
\end{proof}

\begin{prop}\label{P:A3D}
	Define the matrix $A$ (locally) by
	\begin{align}\label{e:A3D}
		A = A_1 + A_2, \quad
		A_j := Y_j \otimes Y_j.
	\end{align}
	Then $A \in C^\al$ and satisfies \cref{e:AGoal}.
\end{prop}

\begin{remark}
	This form of $A$ only applies when $\set{Y_1, Y_2, Y_1 \times Y_2}$ form an orthonormal
	frame in the standard orientation. An expression for $A$ in terms of more general $Y_1, Y_2$
	would need to incorporate the map, $\Cal{G}$---as \cref{e:AOmega2D} does for 2D.
\end{remark}

\begin{proof}[\textbf{Proof of \cref{P:A3D}}]
What we must show is that $A$ as given in \cref{e:A3D} satisfies \cref{e:AGoal}.
For any vorticity, $\vec{\omega} = a_1 Y_1 + a_2 Y_2 + a_3 Y_1 \times Y_2$, we can write
\begin{align*}
	\Omega 
		&= a_1 \Omega_1
			+ a_2 \Omega_2
			+ a_3 \Omega_3,
\end{align*}
where $\Omega_1 = Q(Y_1)$, $\Omega_2 = Q(Y_2)$, $\Omega_3 = Q(Y_1 \times Y_2)$. It follows immediately from \cref{e:CrossProductEquiv} that
\begin{align}\label{e:SimpleProducts}
	\Omega_1 Y_1
		= \Omega_2 Y_2
		= \Omega_3 (Y_1 \times Y_2)
		= 0, \quad
		\Omega_1 Y_2 = - \Omega_2 Y_1.
\end{align}

We first prove $\cref{e:AGoal}_1$. Writing, $\vec{\omega} = a_1 Y_1 + a_2 Y_2 + a_3 Y_1 \times Y_2$, we have,
\begin{align*}
	\Omega (Y_1 \times Y_2)
		&= a_1 \Omega_1 (Y_1 \times Y_2)
			+ a_2 \Omega_2 (Y_1 \times Y_2),
\end{align*}
where the $a_3$ term disappeared by \cref{e:SimpleProducts}.

Now,
\begin{align*}
	\Omega_1 (Y_1 \times Y_2)
		&=
		\spaciousmatrixthree
			{0 & Y_1^3 & - Y_1^2}
			{-Y_1^3 & 0 & Y_1^1}
			{Y_1^2 & - Y_1^1 & 0}
		\spaciousthreevector
			{Y_1^2 Y_2^3 - Y_2^2 Y_1^3}
			{Y_2^1 Y_1^3 - Y_1^1 Y_2^3}
			{Y_1^1 Y_2^2 - Y_2^1 Y_1^2} \\
		&=
		\spaciousmatrixthree
			{(Y_1^3)^2 Y_2^1 - Y_1^3 Y_1^1 Y_2^3 - Y_1^2 Y_1^1 Y_2^2 + (Y_1^2)^2 Y_2^1}
			{-Y_1^3 Y_1^2 Y_2^3 + (Y_1^1)^2 Y_2^2 + (Y_1^1)^2 Y_2^2 - Y_1^1 Y_2^2 Y_1^2}
			{(Y_1^2)^2 Y_2^3 - Y_1^2 Y_2^2 Y_1^3 - Y_1^1 Y_2^1 Y_1^3 + (Y_1^1)^2 Y_2^3}.
\end{align*}
Each of these components simplifies. We have
\begin{align*}
	\brac{\Omega_1 (Y_1 \times Y_2)}^1
		&= \brac{\abs{Y_1}^2 - (Y_1^1)^2} Y_2^1 - Y_1^1(Y_1^3 Y_2^3 + Y_1^2 Y_2^2)
		= Y_2^1.
\end{align*}
Similarly,
\begin{align*}
	\brac{\Omega_1 (Y_1 \times Y_2)}^2
		= Y_2^2, \quad
	\brac{\Omega_1 (Y_1 \times Y_2)}^3
		= Y_2^3.
\end{align*}
We conclude that
\begin{align} \label{e:OmegaY1Y2}
	\Omega_1 (Y_1 \times Y_2)
		&= Y_2, \quad
	\Omega_2 (Y_1 \times Y_2)
		= - Y_1, 
\end{align}
the latter following from symmetry by transposing $Y_1$ and $Y_2$ and using $Y_2 \times Y_1 = - Y_1 \times Y_2$. Thus, by linearity,
\begin{align*}
	\Omega (Y_1 \times Y_2)
		&= a_1 Y_2
		 	- a_2 Y_1, \quad
	A \Omega (Y_1 \times Y_2)
		= a_1 A Y_2
		 	- a_2 A Y_1.
\end{align*}
Noting that we can also write $A_j$ in the form,
\begin{align}\label{e:ARowVectorForm}
	A_j
		= \threevector{Y_j^1 \, Y_j}{Y_j^2 \, Y_j}{Y_j^3 \, Y_j},
\end{align}
each row of $A_j$ being a row vector, we see that
\begin{align*}
	A_j Y_k
		= \spaciousthreevector
		{Y_j^1 Y_j \cdot Y_k}
		{Y_j^2 Y_j \cdot Y_k}
		{Y_j^3 Y_j \cdot Y_k}
		= (Y_j \cdot Y_k)  Y_j.
\end{align*}
Hence,
\begin{align*}
	A_1 Y_1 &= Y_1, \quad
	A_1 Y_2 = 0, \quad
	A_2 Y_1 = 0, \quad
	A_2 Y_2 = Y_2,
\end{align*}
so that
\begin{align} \label{e:RelationAY}
	A Y_1
		&= Y_1, \quad
	A Y_2
		= Y_2,
\end{align}
and hence,
\begin{align}\label{e:AOmegaY1Y2}
	A \Omega (Y_1 \times Y_2)
		&= a_1 Y_2 - a_2 Y_1
		= \Omega(Y_1 \times Y_2).
\end{align}
This establishes $\cref{e:AGoal}_1$.

We next prove $\cref{e:AGoal}_2$ and $\cref{e:AGoal}_3$. By \cref{e:SimpleProducts} and \cref{e:OmegaY1Y2},
\begin{align*}
	& P_{\spn\set{Y_1, Y_2}} \Omega Y_1 = P_{\spn\set{Y_1, Y_2}}  \left(a_2 Y_{2}\times Y_1 +a_{3}Y_{2}\right)=a_{3}Y_{2}, \\
	& P_{\spn\set{Y_1, Y_2}} \Omega Y_2 = P_{\spn\set{Y_1, Y_2}}  \left(a_1 Y_1 \times Y_2 -a_{3}Y_{1}\right)=-a_{3}Y_{1}.
\end{align*}
By \cref{e:RelationAY}, we obtain $\cref{e:AGoal}_2$ and $\cref{e:AGoal}_3$.
\end{proof}

\begin{proof}[\textbf{Proof of \cref{T:A}} in 3D]
	The result follows locally in space from \cref{L:LocalACalpha,P:A3D}. To obtain the result in all
	space, we apply a partition of unity as in the 2D proof in \cref{S:MatrixA}.
\end{proof}

\begin{remark}
	This same approach could be used to prove the 2D result, though it would be longer than
	our approach in \cref{S:MatrixA}, which employed \cref{C:graduCorCor2D}. The proof in this
	section, however, emphasizes that \cref{T:A} is almost purely geometric in nature.
\end{remark}

We used the following lemma above:

\begin{lemma}\label{L:graduTCalpha}
	For $d = 3$,
	$(\grad u)^T (Y_1 \times Y_2) \in C^\al$
	with
	\begin{align*}
		\norm{(\grad u)^T (Y_1 \times Y_2)}_{C^\al}
			\le \max_{j = 1, 2} \norm{Y_j \cdot \grad u}_{C^\al}
				\max_{j = 1, 2} \norm{Y_j}_{C^\al}.
	\end{align*}
\end{lemma}

\begin{proof}
	We have,
	\begin{align*}
		(\grad u)^T (Y_1 \times Y_2)
			=
			\matrixthree
				{\prt_1 u^1 & \prt_1 u^2 & \prt_1 u^3}
				{\prt_2 u^1 & \prt_2 u^2 & \prt_2 u^3}
				{\prt_3 u^1 & \prt_3 u^2 & \prt_3 u^3}
			\matrixthree
				{Y_1^2 Y_2^3 - Y_1^3 Y_2^2}
				{Y_1^3 Y_2^1 - Y_1^1 Y_2^3}
				{Y_1^1 Y_2^2 - Y_1^2 Y_2^1}.
	\end{align*}
	We will write out only first component in detail, the other two components being very
	similar. Multiplying, we have
	\begingroup
	\allowdisplaybreaks
	\begin{align*}
		[(\grad &u)^T (Y_1 \times Y_2)]^1 \\
			&= \prt_1 u^1 (Y_1^2 Y_2^3 - Y_1^3 Y_2^2)
				+ \prt_1 u^2 (Y_1^3 Y_2^1 - Y_1^1 Y_2^3)
				+ \prt_1 u^3 (Y_1^1 Y_2^2 - Y_1^2 Y_2^1) \\
			&= Y_2^1 (\prt_1 u^2 Y_1^3 - \prt_1 u^3 Y_1^2)
				+ Y_2^2 (-\prt_1 u^1 Y_1^3 + \prt_1 u^3 Y_1^1)
				+ Y_2^3 (\prt_1 u^1 Y_1^2 - \prt_1 u^2 Y_1^1) \\
			&= Y_2^1 (\prt_1 u^2 Y_1^3 - \prt_1 u^3 Y_1^2)
				+ Y_2^2 ((\prt_2 u^2 + \prt_3 u^3) Y_1^3 + \prt_1 u^3 Y_1^1) \\
				&\qquad\qquad\qquad
				+ Y_2^3 ((-\prt_2 u^2 - \prt_3 u^3) Y_1^2 - \prt_1 u^2 Y_1^1) \\
			&= Y_2^1 (\prt_1 u^2 Y_1^3 - \prt_1 u^3 Y_1^2)
				+ Y_2^2 ((\prt_1 u^3 Y_1^1 + \prt_2 u^3 Y_1^2 + \prt_3 u^3 Y_1^3)
					- \prt_2 u^3 Y_1^2 + \prt_2 u^2 Y_1^3) \\
				&\qquad\qquad\qquad
				+ Y_2^3 (-(\prt_1 u^2 Y_1^1 + \prt_2 u^2 Y_1^2 + \prt_3 u^2 Y_1^3)
					+ \prt_2 u^2 Y_1^3 - \prt_3 u^3 Y_1^2) \\
			&= Y_2^1 (\prt_1 u^2 Y_1^3 - \prt_1 u^3 Y_1^2)
				+ Y_2^2 (\grad u^3 Y_1
					- \prt_2 u^3 Y_1^2 + \prt_2 u^2 Y_1^3) \\
				&\qquad\qquad\qquad
				+ Y_2^3 (-\grad u^2 Y_1
					+ \prt_2 u^2 Y_1^3 - \prt_3 u^3 Y_1^2) \\
			&= Y_2^2 \grad u^3 Y_1 - Y_2^3 \grad u^2 Y_1
				+ \prt_1 u^2 Y_1^3 Y_2^1 - \prt_1 u^3 Y_1^2 Y_2^1
				- \prt_2 u^3 Y_1^2 Y_2^2 + \prt_2 u^2 Y_1^3 Y_2^2 \\
				&\qquad\qquad\qquad
				+ \prt_2 u^2 Y_1^3 Y_2^3 - \prt_3 u^3 Y_1^2 Y_2^3 \\
			&= Y_2^2 \grad u^3 Y_1 - Y_2^3 \grad u^2 Y_1
				+ Y_1^3 (\prt_1 u^2 Y_2^1 + \prt_2 u^2 Y_2^2 + \prt_2 u^2 Y_2^3) \\
				&\qquad\qquad\qquad
				- Y_1^2 (\prt_1 u^3 Y_2^1 + \prt_2 u^3 Y_2^2 + \prt_3 u^3 Y_2^3) \\
			&= Y_2^2 \grad u^3 Y_1 - Y_2^3 \grad u^2 Y_1
				+ Y_1^3 \grad u^2 Y_2 - Y_1^2 \grad u^3 Y_2 \\
			&= Y_2^2 (Y_1 \cdot \grad u)^3 - Y_2^3 (Y_1 \cdot \grad u)^2
				+ Y_1^3 (Y_2 \cdot \grad u)^2 - Y_1^2 (Y_2 \cdot \grad u)^3
			\in C^\al.
	\end{align*}
	\endgroup
\end{proof}

\begin{remark}
	\cref{T:A} has a clear extension to all dimensions $d \ge 2$. It is the
	computation of the analogous bound to that in \cref{L:graduTCalpha} that
	complicates the general-dimensional proof.
\end{remark}

\section{Approximate solutions and transport equations}\label{S:Transport}

\noindent Having established \cref{T:Equivalence}, \cref{T:FamilyVel} follows immediately from \cref{T:FamilyVelOmega2D,T:FamilyVelOmega3D}. We now, however, begin the presentation of a (nearly) self-contained proof of \cref{T:FamilyVelOmega2D} using elementary methods, as promised in in \cref{S:Introduction}, inspired by Serfati's \cite{SerfatiVortexPatch1994}. (We outline the changes to this proof needed to obtain \cref{T:FamilyVelOmega3D} in \cref{S:3DOutline}.)

We start in this section with a mollification of the initial data so we can work with smooth solutions, and then discuss the various transport equations that enter into the proof.

We regularize the initial data by setting $u_{0, \eps} = \rho_\eps * u_0$, where $\rho_\eps$ is the standard mollifier of \cref{D:Mollifier}, letting $\eps$ range over values in $(0, 1]$. It follows that $\omega_{0, \eps} = \rho_\eps * \omega_0$. Then there exists a smooth solution, $\omega_\eps(t) \in C^\iny(\R^2)$, to the Euler equations in vorticity form, \cref{e:VorticityE}, \cref{e:BSLaw}, for all time with $C^\iny$ velocity field, $u_\eps$ (\cite{Leray1933, Wolibner1933} or see Theorem 4.2.4 of \cite{C1998}). These solutions converge to a weak solution $\omega(t)$ of \cref{e:VorticityE}, \cref{e:BSLaw}. (We say more about convergence in \cref{S:Convergence}.)

The flow map, $\eta_\eps$, is given in \cref{e:etaDef} with $u_\eps$ in place of $u$. Moreover, all the $L^p$-norms of $\omega_\eps$ are conserved over time with
\begin{align}\label{e:omegaNormp}
    \norm{\omega_\eps(t)}_{L^{p}}
        = \norm{\omega_{\eps, 0}}_{L^{p}}
        \le \norm{\omega_0}_{L^{p}}
        \le \norm{\omega_0}_{L^1 \cap L^\iny}
        =: \norm{\omega_0}_{L^1} + \norm{\omega_0}_{L^\iny}
\end{align}
for all $p \in [1, \iny]$. Also,
\begin{align}\label{e:uepsBound}
    \norm{u_\eps(t)}_{L^\iny} \le C \norm{\omega_0}_{L^1 \cap L^\iny}    
\end{align}
(see Proposition 8.2 of \cite{MB2002}) so $\norm{u_\eps}_{L^\iny(\R \times \R^2)}$ is uniformly bounded in $\eps$.

For most of the proof we will use these smooth solutions, passing to the limit as $\eps \to 0$ in the final steps in \cref{S:Convergence}.

Let $Y_0 \in C^\al$ with $\dv Y_0 \in C^\al$. We let
\begin{align}\label{e:Yeps}
    Y_\eps(t, \eta_\eps(t, x)) = Y_0(x) \cdot \nabla \eta_\eps(t, x)
\end{align}
be the pushforward of $Y_0$ under the flow map $\eta_\eps$, as in \cref{e:PushForward}.
(Because $Y_0$ has all the regularity we need, it would be counterproductive to mollify it, as we do the initial data.)
Similarly, we define the pushforward of the family $\Cal{Y}_0$ of \cref{T:FamilyVel} as in \cref{e:YFamily}, by
\begin{align}\label{e:YepsFamily}
    \Cal{Y}_\eps(t) = (Y_\eps^{(\la)}(t))_{\la \in \Lambda}, \quad
        Y_\eps^{(\la)}(t, \eta(t, x)) := (Y_0^{(\la)}(x) \cdot \nabla) \eta_\eps(t, x).
\end{align}
 (Note the slight notational collision between $Y_\eps$ and $Y_0$, $\Cal{Y}_\eps$ and $\Cal{Y}_0$, and $\omega_\eps$ and $\omega_0$; this should not, however, cause any confusion.)

For the remainder of this section we focus on one element, $Y_0 \in \Cal{Y}_0$.

Standard calculations show that
\begin{align}\label{e:YTransportAlmost}
    \prt_t Y_\eps + u_\eps \cdot \nabla Y_\eps
        = Y_\eps \cdot \nabla u_\eps
\end{align}
and that
\begin{align}\label{e:YTransportDiv}
    \begin{split}
        \prt_t \dv Y_\eps + u_\eps \cdot \grad \dv Y_\eps &= 0, \\
        \prt_t \dv (\omega_\eps Y_\eps)
            + u_\eps \cdot \grad \dv (\omega_\epsilon Y_\eps) &= 0,
    \end{split}
\end{align}
the latter equality using that the vorticity is transported by the flow map.
Hence,
\begin{align}\label{e:divomegaY}
    \begin{split}
		\dv Y_\eps(t, x) = \dv Y_0(\eta_\eps^{-1}(t, x)), \\
        \dv (\omega_\eps Y_\eps)(t, x) = \dv (\omega_{0, \eps} Y_0)(\eta_\eps^{-1}(t, x)).
    \end{split}
\end{align}

\begin{remark}\label{R:WeakTransport}
Actually, the transport equations in \cref{e:YTransportAlmost,e:YTransportDiv}, and others we will state later, are satisfied in a weak sense.
We refer to Definition 3.13 of \cite{BahouriCheminDanchin2011} for the notion of weak transport. With the exception of the use of Theorem 3.19 of \cite{BahouriCheminDanchin2011} in the proof of \cref{L:RegularitywY}, we will treat all transport equations as though they are satisfied in a strong sense, however, justifying such use in \cref{A:TransportEstimates}.
\end{remark}

We can also write \cref{e:YTransportAlmost} and \cref{e:YTransportDiv} as
\begin{align}\label{e:TwoTransports}
    \begin{split}
        &\diff{}{t} Y_\eps(t, \eta_\eps(t, x))
            = (Y_\eps \cdot \nabla u_\eps)(t, \eta_\eps(t, x)), \\
        &\diff{}{t} \dv (\omega_\eps Y_\eps)(t, \eta_\eps(t, x))
            = 0.
    \end{split}
\end{align}

Define the vector field
\begin{align}\label{e:R0Def}
    R_{0, \eps}
        = \omega_{0, \eps} Y_0
            + \rho_\eps * \nabla \Cal{F}_2 * \dv (\omega_0 Y_0)
            - \rho_\eps * \left(\omega_0 Y_0\right)
\end{align}
at the initial time only, and observe that
\begin{align*}
    \dv R_{0, \eps}
        &= \dv (\omega_{0, \eps} Y_0)
            + \dv \bigpr{\rho_\eps * \pr{\nabla \Cal{F}_2
            * \dv \left(\omega_0 Y_0\right)
            - \omega_0 Y_0}}
        = \dv (\omega_{0, \eps} Y_0),
\end{align*}
where we used that $\Delta \Cal{F}_2$ is the Dirac delta function.

\begin{lemma}\label{L:R0Decomp}
The vector field $R_{0, \eps}$, defined in \cref{e:R0Def}, is in $C^\alpha(\R^2)$, with
\begin{align*}
    \norm{R_{0, \eps}}_{C^\al} \le C_\al,
\end{align*}
uniformly over $\eps$ in $(0, 1]$, where $C_\al$ is as in \cref{e:Calpha}.
\end{lemma}

\begin{proof}
We rewrite $R_{0, \eps}$ in the form,
\begin{align*}
    R_{0, \eps}
        = \rho_\eps * \nabla \Cal{F}_2 * \dv(\omega_0 Y_0)
            +\bigbrac{(\rho_\eps * \omega_0)Y_0 - \rho_\eps * (\omega_0 Y_0)}.
\end{align*}
Since $\displaystyle \nabla \Cal{F}_2 * \dv (\omega_0 Y_0)\in C^\alpha(\R^2)$ by \cref{P:EquivalentConditions} (noting that $\omega_0 Y_0 \in L^1 \cap L^\iny$), we have 
\begin{align}\label{e:R0InternalBound}
    \norm{\rho_\eps * \nabla \Cal{F}_2 * \dv(\omega_0 Y_0)}_{C^{\alpha}}
        \le C \norm{\nabla \Cal{F}_2 * \dv(\omega_0 Y_0)}_{C^{\alpha}}
        \le C(\omega_0, Y_0).
\end{align}
Since $Y_0 \in C^\alpha(\R^2)$, applying \cref{L:SerfatiLemma2Inf} with the kernel $L_1$ of \cref{L:SerfatiKernels}, we have
\begin{align}\label{e:rhoepsY0Est}
    \begin{split}
    &\norm{(\rho_\eps * \omega_0)Y_0 - \rho_\eps * (\omega_0 Y_0)}_{C^\al}
        = \norm{\int_{\R^2} \rho_\eps(x - y)\omega_0(y) \brac{Y_0(x)-Y_0(y)}
                \, dy}_{C^\al} \\
        &\qquad
        \le C(\omega_0, Y_0) \pr{\al^{-1} (1 - \al)^{-1}}
        = C_\al.
    \end{split}
\end{align}
This completes the proof. 
\end{proof}

Finally, we prove the propagation of regularity of $\dv (\omega_\eps Y_\eps)$.

\begin{lemma}\label{L:RegularitywY}
We have $\dv (\omega_\eps Y_\eps)(t) \in C^{\alpha-1}(\R^2)$ with 
\begin{align*}
    \norm{\dv (\omega_\eps Y_\eps)(t)}_{C^{\alpha-1}}
        \le C_\al
            \exp\int_0^t \norm{\nabla u_\eps(s)}_{L^\iny} \, ds.
\end{align*}
\end{lemma}

\begin{proof}
We first obtain the following bound:
\begin{align*}
    \norm{\dv (\omega_\eps Y_\eps)(t)}_{C^{\alpha-1}}
        \le C \norm{\dv (\omega_{0, \eps} Y_0)}_{C^{\alpha-1}}
            \exp\int_0^t \norm{\nabla u_\eps(s)}_{L^\iny} \, ds.
\end{align*}
The details of obtaining this bound is presented in the appendix (see \cref{P:Holder estimate of transport equation}). 

We must still, however, bound $\norm{\dv (\omega_{0, \eps} Y_0)}_{C^{\alpha-1}}$ uniformly in $\eps$. From the triangle inequality,
\begin{align*}
    \norm{\dv (\omega_{0, \eps} Y_0)}_{C^{\alpha-1}}
        \le \norm{\dv(\omega_{0, \eps} Y_0) - \rho_\eps * \dv (\omega_0 Y_0)}_{C^{\al - 1}}
            + \norm{\rho_\eps * \dv (\omega_{0, \eps} Y_0)}_{C^{\alpha-1}}.
\end{align*}
Now,
\begin{align*}
    \norm{\dv(\omega_{0, \eps} Y_0) - \rho_\eps * \dv (\omega_0 Y_0)}_{C^{\al - 1}}
        \le \norm{\omega_{0, \eps} Y_0 - \rho_\eps * (\omega_0 Y_0)}_{C^\al}
        \le C_\al,
\end{align*}
the first inequality following from \cref{e:dvvCalBound}, the second from 
\cref{e:rhoepsY0Est}.
Also,
\begin{align*}
    \norm{\rho_\eps * \dv (\omega_0 Y_0)}_{C^{\al - 1}}
        &\le C\norm{\grad \Cal{F}_2 * (\rho_\eps * \dv (\omega_0 Y_0))}_{C^\al}
        = C\norm{\rho_\eps * (\grad \Cal{F}_2 * \dv (\omega_0 Y_0))}_{C^\al} \\
        &\le C\norm{\grad \Cal{F}_2 * \dv (\omega_0 Y_0)}_{C^\al}
        \le C \pr{\norm{\omega_0 Y_0}_{L^1 \cap L^\iny}
            + \norm{\dv (\omega_0 Y_0)}_{C^{\al - 1}}}.
\end{align*}
For the first inequality we applied \cref{P:EquivalentConditions}, for the second inequality we used $\norm{\rho_\eps * f}_{C^\al} \le \norm{f}_{C^\al}$, and for the third we applied \cref{P:EquivalentConditions} once more.
Hence,
\begin{align*}
    \norm{\dv (\omega_{0, \eps} Y_0)}_{C^{\alpha-1}}
        \le C_\al
            + \pr{\norm{\omega_0 Y_0}_{L^1 \cap L^\iny}
                + \norm{\dv (\omega_0 Y_0)}_{C^{\al - 1}}}
        \le C_\al.
\end{align*}
\end{proof}

\Ignore{ 
\begin{remark}\label{R:CannotMollifyY0}
    It would be natural to let $Y_{0, \eps} = \rho_\eps * Y_0$ and pushforward
    $Y_{0, \eps}$ rather than $Y_0$ in the definition of $Y_\eps$.
    This would allow us to use transport equations purely in strong form.
    It is the bound in \cref{e:R0InternalBound}, however, that prevents us from
    doing this, as the equivalent bound with $Y_{0, \eps}$ in place of $Y_0$
    may not hold true. Instead, we take the approach described in
    \cref{A:TransportEstimates}.  
\end{remark}
} 

\begin{remark}
	It is easy to see that the estimates in \cref{L:R0Decomp,L:RegularitywY} apply
	equally well to the whole family, $\Cal{Y}$.
\end{remark}

\section{
Propagation of striated regularity of vorticity in 2D} 
\label{S:ProofOfSerfatisResult}

\noindent
Before proceeding to the fairly long and technical proof of \cref{T:FamilyVelOmega2D}, let us first present the overall strategy.

We start in \cref{S:EstimationGraduAndY} by bounding $\norm{\grad u_\eps(t)}_{L^\iny}$ above by the quantity,
\begin{align}\label{e:VepsDef}
        V_\eps(t)
            :=
            \norm{\omega_0}_{L^\iny}
                + \norm{\PV \int \grad K(\cdot - y) \omega_\eps(t, y) \, dy}_{L^\iny}.
\end{align}
We also bound the gradients of the flow map and inverse flow map in terms of $V_\eps(t)$. These estimates are entirely classical and do not involve $\Cal{Y}_\eps$.

In \cref{S:EstYR}, we bound $\norm{Y_\eps}_{C^\al}$ in terms of $V_\eps(t)$ and $\norm{K * \dv(\omega_\eps Y_\eps)}_{C^\al}$.
This gives us a bound on $\norm{Y_\eps}_{C^\al}$ in terms of $V_\eps(t)$ alone. We also develop a pointwise bound from below of $\abs{Y_\eps}(t, x)$ in terms of $V_\eps(t)$.

In \cref{S:Refinedgradueps}, we bound $V_\eps(t)$ in terms of $\norm{Y_\eps}_{C^\al}$. Here, we make great use of Serfati's linear algebra lemma, \cref{L:SerfatiLemma1}. We also need the pointwise bound from below of $\abs{Y_\eps}(t, x)$ developed in \cref{S:EstYR}, for $\abs{Y_\eps}$ appears in the denominator in our estimates. The end result is a bound on $V_\eps(t)$ in terms of itself that will allow us to close the estimates and so apply Gronwall's lemma to bound $V_\eps(t)$.

The bound on $\norm{Y_\eps}_{C^\al}$ in terms of $V_\eps(t)$ in \cref{S:Refinedgradueps} also involves $\norm{K * \dv(\omega_\eps Y_\eps)}_{C^\al}$, but this is bounded in terms of $V_\eps(t)$ easily by \cref{L:RegularitywY,P:EquivalentConditions}. This, in turn, yields the bounds on all the other quantities, as in \crefrange{e:MainBoundsgradu}{e:IYBound}.

It remains, however, to show that the sequence of approximate solutions converge to a solution in a manner such that \crefrange{e:MainBoundsgradu}{e:IYBound} hold. Such an argument is given in \cite{C1998}; we restrict ourselves to describing in \cref{S:Convergence} the role that assuming $\dv \Cal{Y}_0 \in C^\al$ plays in the convergence argument, for this is a somewhat subtle point.

\subsection{Preliminary estimate of $\norm{\nabla u_\eps(t)}_{L^\iny}$, $\norm{\nabla \eta_\eps(t)}_{L^\iny}$, and $\norm{\nabla \eta^{-1}_\eps(t)}_{L^\iny}$}\label{S:EstimationGraduAndY}

By the expression for $\grad u_\eps$ in \cref{P:graduExp}, and using \cref{e:omegaNormp}, we have,
\begin{align*}
    \norm{\grad u_\eps(t)}_{L^\iny}
        \le V_\eps(t).
\end{align*}

As in \cref{e:etaDef}, the defining equation for $\eta_\eps$ is
\begin{align}\label{e:etaDefRedux}
    \prt_t \eta_\eps(t, x)
        = u_\eps(t, \eta_\eps(t, x)),
            \quad
        \eta_\eps(0,x) = x,
\end{align}
or, in integral form,
\begin{align}\label{e:etaepsIntForm}
    \eta_\eps(t, x) = x + \int_0^t u_\eps(s, \eta_\eps(s, x)) \, ds.
\end{align}
This immediately implies that 
\begin{align}\label{e:gradetaBound}
    \norm{\nabla \eta_\eps(t)}_{L^\iny}
        \le \exp \int_0^t V_\eps(s) \, ds. 
\end{align}
Similarly,
\begin{align}\label{e:gradetaInvBound}
    \norm{\nabla \eta^{-1}_\eps(t)}_{L^\iny}
        \le \exp \int_0^t V_\eps(s) \, ds. 
\end{align}
The bound in \cref{e:gradetaInvBound} does not follow as immediately as that in \cref{e:gradetaBound} because the flow is not autonomous. For the details, see, for instance, the proof of Lemma 8.2 p. 318-319 of \cite{MB2002}.

\subsection{Estimate of $Y_\eps$}\label{S:EstYR}

Taking the inner product of \cref{e:TwoTransports}$_1$ with $Y_\eps(t, \eta_\eps(t, x))$ gives
\begin{align*}
    \diff{}{t} Y_\eps(t, \eta_\eps(t, x)) \cdot Y_\eps(t, \eta_\eps(t, x))
        = (Y_\eps \cdot \grad u_\eps)(t, \eta_\eps(t, x)) \cdot Y_\eps(t, \eta_\eps(t, x)).
\end{align*}
The left-hand side equals
\begin{align*}
    \frac{1}{2} \diff{}{t} \abs{Y_\eps(t, \eta_\eps(t, x))}^2
\end{align*}
so
\begin{align*}
    \abs{\diff{}{t} \abs{Y_\eps(t, \eta_\eps(t, x))}^2}
        &\le 2 \norm{\grad u_\eps(t, \eta_\eps(t, \cdot))}_{L^\iny}
            \abs{Y_\eps(t, \eta_\eps(t, x))}^2 \\
        &= 2 \norm{\grad u_\eps(t)}_{L^\iny}
            \abs{Y_\eps(t, \eta_\eps(t, x))}^2
        \le 2 V_\eps(t)
            \abs{Y_\eps(t, \eta_\eps(t, x))}^2.
\end{align*}
It follows that 
\begin{align*}
    \diff{}{t} \abs{Y_\eps(t, \eta_\eps(t, x))}^2
        \le 2 V_\eps(t)
            \abs{Y_\eps(t, \eta_\eps(t, x))}^2.
\end{align*}
Similarly, 
\begin{align*}
    \diff{}{t} \abs{Y_\eps(t, \eta_\eps(t, x))}^2
       \ge - 2 V_\eps(t)
            \abs{Y_\eps(t, \eta_\eps(t, x))}^2.
\end{align*}
Integrating in time and applying \cref{L:Gronwall} gives
\begin{align*}
    \abs{Y_0(x)}
            e^{- \int_0^t \norm{\grad u_\eps(s)}_{L^\iny} \, ds}
        \le \abs{Y_\eps(t, \eta_\eps(t, x))}
        \le \abs{Y_0(x)}
            e^{\int_0^t \norm{\grad u_\eps(s)}_{L^\iny} \, ds}.
\end{align*}
We conclude that
\begin{align}\label{e:YepsBelow}
    \abs{Y_\eps(t, \eta_\eps(t, x))}
        \ge \abs{Y_0(x)} e^{-\int_0^t V_{\epsilon}(s) \, ds}
\end{align}
and taking the $L^\iny$ norm in $x$ that
\begin{align}\label{e:YepsLinfAbove}
    \norm{Y_\eps(t)}_{L^\iny}
        \le \norm{Y_0}_{L^\iny} e^{\int_0^t V_{\epsilon}(s) \, ds}.
\end{align}

Integrating \cref{e:TwoTransports}$_1$ in time and substituting $\eta^{-1}_\eps(t, x)$ for $x$ yields
\begin{align}\label{e:YInt}
    Y_\eps(t, x)
        = Y_0(\eta^{-1}_\eps(t, x))
            + \int_0^t (Y_\eps \cdot \grad u_\eps) (s, \eta_\eps(s,\eta^{-1}_\eps(t, x))) \, ds. 
\end{align}
Taking the $\dot{C}^\alpha$ norm and applying \cref{e:CalphaFacts}$_1$, we have 
\begin{align*} 
    \norm{Y_\eps(t)}_{\dot{C}^\alpha}
        &\le \norm{Y_0}_{\dot{C}^\alpha}
            \norm{\nabla \eta^{-1}_\eps(t)}^\alpha_{L^\iny}
            + \int_0^t \norm{(Y_\eps \cdot \nabla u_\eps)(s)}_{\dot{C}^\alpha}
            \norm{\nabla(\eta_\eps(s,\eta^{-1}_\eps(t, x)))}^\alpha_{L^\iny} \, ds.
\end{align*}

Now, by \cref{C:graduCor}, we have
\begin{align*} 
    Y_\eps \cdot \nabla u_\eps(s, x)
        = \PV \int
                &\grad K(x - y) \omega_\eps(s, y)
                \brac{Y_\eps(s, x)-Y_\eps(s, y)} \, dy \\
            &+ K * \dv (\omega_\eps Y_\eps)(s, x)
        =: \text{I} + \text{II}
\end{align*}
with
\begin{align*}
    \norm{\text{I}}_{C^\alpha}
        &\le C \norm{Y_\eps(s)}_{C^\alpha} V_\eps(s).
\end{align*}
By \cref{P:EquivalentConditions,L:RegularitywY}, we have
\begin{align*}
    \norm{\text{II}}_{C^\alpha}
        \le C_\al \exp\int_0^s V_\eps(\tau) \, d \tau.
\end{align*}
It follows that
\begin{align}\label{e:YgraduepsBound}
    \norm{Y_\eps \cdot \grad u_\eps(t)}_{C^\al}
        \le \norm{Y_\eps(t)}_{C^\alpha} V_\eps(t)
            + C_\al \exp\int_0^t V_\eps(\tau) \, d \tau.
\end{align}

To estimate $\smallnorm{\nabla(\eta_\eps(s,\eta^{-1}_\eps(t, x)))}_{L^\iny}$, we start with
\begin{align*}
    \prt_\tau \eta_\eps(\tau, \eta^{-1}_\eps(t, x))
        = u_\eps(\tau, \eta_\eps(\tau, \eta^{-1}_\eps(t, x))),
\end{align*}
which follows from \cref{e:etaDefRedux}. Applying the spatial gradient and the chain rule gives
\begin{align*}
    \prt_\tau \grad \pr{\eta_\eps(\tau, \eta^{-1}_\eps(t, x))}
        = \nabla u_\eps(\tau, \eta_\eps(\tau, \eta^{-1}_\eps(t, x)))
            \nabla(\eta_\eps(\tau, \eta^{-1}_\eps(t, x))).
\end{align*}
Integrating in time and using
$
    \nabla(\eta_\eps(\tau, \eta^{-1}_\eps(t, x)))|_{\tau = t}
        = I^{2\times 2},
$
the identity matrix, we have
\begin{align*}
    \grad \pr{\eta_\eps(s, \eta^{-1}_\eps(t, x))}
            &= I^{2 \times 2}
            - \int_s^t \nabla u_\eps(\tau, \eta_\eps(\tau, \eta^{-1}_\eps(t, x)))
            \nabla(\eta_\eps(\tau, \eta^{-1}_\eps(t, x))) \, d \tau.
\end{align*}

By \cref{L:Gronwall}, then,
\begin{align*}
    \norm{\nabla(\eta_\eps(s,\eta^{-1}_\eps(t, x)))}_{L^\iny}
            \le \exp\int_s^t \norm{\nabla u_\eps(\tau)}_{L^\iny} \, d\tau
        \le \exp\int_s^t V_\eps(\tau) \, d\tau.
\end{align*}

These bounds with \cref{e:gradetaInvBound}, and accounting for \cref{e:YepsLinfAbove}, give
\begin{align*} 
    \begin{split}
        &\norm{Y_\eps(t)}_{C^\alpha}
            \le \norm{Y_0}_{C^\al} \exp \pr{\al \int_0^t V_\eps(s) \, ds} \\
        &\qquad\qquad
                + \int_0^t \brac{\norm{Y_\eps(s)}_{C^\alpha} V_\eps(s)
                    + C_\al \exp\int_0^s V_\eps(\tau) \, d \tau
                    }
            \exp \pr{\al \int_s^t V_\eps(\tau) \, d\tau} \, ds
            \\
        &\qquad\le (\norm{Y_0}_{C^\al} + C_\al t) \exp\int_0^t V_\eps(s) \, ds
            + \int_0^t \norm{Y_\eps(s)}_{C^\alpha} V_\eps(s)
                \brac{\exp \int_s^t V_\eps(\tau) \, d\tau} \, ds.
    \end{split}
\end{align*}
Letting
\begin{align*}
    y_\eps(t)
        = \norm{Y_\eps(t)}_{C^\alpha} \exp\brac{-\int_0^t V_\eps(s) \, ds}
\end{align*}
it follows that $y_\eps$ satisfies the inequality, 
\begin{align*}
    y_\eps(t)
        \le \norm{Y_0}_{C^\al} + C_\al t + \int_0^t V_\eps(s) y_\eps(s) \, ds.
\end{align*}
Therefore, by \cref{L:Gronwall}, we obtain
\begin{align*} 
    y_\eps(t)
        &\le \pr{\norm{Y_0}_{C^\al} + C_\al t}
            \exp \pr{\int_0^t V_\eps(s) \, ds}
        \le C_\al (1 + t)
            \exp \pr{\int_0^t V_\eps(s) \, ds}
\end{align*}
and thus,
\begin{align}\label{e:YRCalphaBound}
    \norm{Y_\eps(t)}_{C^\alpha}
        \le C_\al (1 + t)
            \exp \pr{2 \int_0^t V_\eps(s) \, ds}.
\end{align}

\subsection{Estimate of $V_\eps$}\label{S:Refinedgradueps}
In \cref{P:graduLInf} we bounded $\grad u$ in $L^\iny$ using a bound on $Y \cdot \grad u$ in $L^\iny$. Given \cref{e:YgraduepsBound,e:YRCalphaBound}, we could do the same now for bounding $\grad u_\eps$ in $L^\iny$. The resulting bound, however, would be useless, as it could not be closed. We instead employ \cref{L:SerfatiLemma1} to obtain a more refined estimate of $\grad u_\eps$ in $L^\iny$. (This estimate would not, however, be suited to prove \cref{P:graduLInf}, for it will be an estimate of $V_\eps$ in terms of itself in a manner sufficient to allow all our estimates to be closed using Gronwall's lemma.)

Until the very end of this section, we will estimate quantities at a fixed point, $(t, x) \in (\R \times \R^2)$, though we will generally suppress these arguments for simplicity of notation.

We start by splitting the second term in $V_{\epsilon}$ in \cref{e:VepsDef} into two parts, as
\begin{align}\label{e:pvgradKomegaFirstSplit}
    \begin{split}
        \PV &\int \grad K(x - y) \omega_\eps(t, y) \, dy \\
        	&= \PV \int \grad ((a_r K))(x - y) \omega_\eps(t, y) \, dy
				+ \PV \int \grad ((1 - a_r) K)(x - y) \omega_\eps(t, y) \, dy.
    \end{split}
\end{align}
where $r \in (0, 1]$ will be chosen later (in \cref{e:rChoice}).

On the support of $\grad (1 - a_r) = - \grad a_r$, $\abs{x - y} \le 2 r$, so
\begin{align}\label{e:OnearKBound}
    \abs{\grad ((1 - a_r) K)}
        \le \abs{(1 - a_r) \grad K} + \abs{\grad a_r \otimes K}
        \le C \abs{x - y}^{-2}.
\end{align}
Hence, one term in \cref{e:pvgradKomegaFirstSplit} is easily bounded by 
\begin{align}\label{e:OnemarKomegaBound}
	\begin{split}
    &\abs{\PV \int \grad ((1 - a_r) K)(x - y) \omega_\eps(t, y) \, dy}
        \le C \int_{B_r^C(x)} \abs{x - y}^{-2} \abs{\omega_\eps(t, y)} \, dy \\
        &\qquad
        \le C \int_r^1 \frac{\norm{\omega_\eps}_{L^\iny}}{\rho^2} \rho \, d \rho
        + C \smallnorm{\abs{x - \cdot}^{-2}}_{L^\iny(B_1^C(x))}
            \norm{\omega_{\eps, 0}}_{L^1} \\
        &\qquad
        \le -C \log r \norm{\omega_0}_{L^\iny}
            + C \norm{\omega_0}_{L^1}
        \le C (-\log r + 1) \norm{\omega_0}_{L^1 \cap L^\iny}.
	\end{split}
\end{align}

For the other term in \cref{e:pvgradKomegaFirstSplit}, fix $x \in \R^2$ and choose any $Y_0 \in \Cal{Y}_0$ such that
\begin{align}\label{e:Y0Below}
    \abs{Y_0}\left(\eta^{-1}_\eps(t, x)\right)
        \ge I(\Cal{Y}_0).
\end{align}

Letting $\mu_{rh}$ be as in \cref{D:Radial}, by virtue of \cref{P:ConvEq}, we can write
\begin{align*} 
    \abs{\PV \int \grad ((a_r K))(x - y) \omega_\eps(t, y) \, dy}
        = \abs{\lim_{h \to 0} \grad (\mu_{hr} K) * \omega_\eps(t, x)}
        = \lim_{h \to 0} \abs{B},
\end{align*}
where
\begin{align*}
    B =
    	B(t, x)
		:= \grad \brac{\mu_{rh} \grad \Cal{F}_2} * \omega_\eps.
\end{align*}

Because $\grad \brac{\mu_{rh} \grad \Cal{F}_2}$ is not in $L^1$ uniformly in $h > 0$, we cannot estimate $\abs{B}$ directly. Instead, we will apply \cref{L:SerfatiLemma1} with
\begin{align*}
    M
        :=
        \begin{pmatrix}
            Y_\eps & Y_\eps^\perp
        \end{pmatrix},
\end{align*}
noting that the last column of $M$ is equal to the last column of its cofactor matrix. Hence,
\begin{align*}
	M_1 = Y_\eps, \quad \det M = \abs{Y_\eps}^2.	
\end{align*}

Applying \cref{L:SerfatiLemma1}, we have 
\begin{align*}
    \begin{split}
        \abs{B}
            \le C \frac{P_1(Y_\eps)}{\abs{Y_\eps}^4}
                    \abs{BM_1}
                + \frac{P_2(Y_\eps)}{\abs{Y_\eps}^2} \abs{\tr B}.
    \end{split}
\end{align*}

We now compute $\tr B$. We have,
\begin{align*} 
\begin{split}
    \tr B
        &= \brac{\prt_1 \mu_{rh} \prt_1 \Cal{F}_2} * \omega_\eps
            + \brac{\prt_2 \mu_{rh} \prt_2 \Cal{F}_2} * \omega_\eps
            + \brac{\mu_{rh} \Delta \Cal{F}_2} * \omega_\eps \\
        &= \brac{\prt_1 \mu_{rh} \prt_1 \Cal{F}_2} * \omega_\eps
            + \brac{\prt_2 \mu_{rh} \prt_2 \Cal{F}_2} * \omega_\eps,
\end{split}
\end{align*}
using $\Delta \Cal{F}_2=\delta_0$ and $\mu_{rh}(0)=0$ to remove the last term.

But, referring to \cref{R:Radial}, for $j = 1, 2$, we have
\begin{align*}
    \begin{split}
        \abs{\brac{\prt_j \mu_{rh} \prt_j \Cal{F}_2} * \omega_\eps}
            &\le \frac{C}{r} \int_{r < \abs{x - y} < 2r}
                \frac{\abs{\omega_\eps(t, y)}}{\abs{x - y}} \, dy
            + \frac{C}{h} \int_{h < \abs{x - y} < 2h}
                \frac{\abs{\omega_\eps(t, y)}}{\abs{x - y}} \, dy \\
            &\le \frac{C}{r} \int_r^{2r} \frac{\norm{\omega_\eps(t)}_{L^\iny}}{\rho}
                \rho \, d \rho
            + \frac{C}{h} \int_h^{2h} \frac{\norm{\omega_\eps(t)}_{L^\iny}}{\rho}
                \rho \, d \rho \\
            &= C \norm{\omega_\eps(t)}_{L^\iny}
            \end{split}
\end{align*}
so that
\begin{align*}
    \lim_{h \to 0} \abs{\tr B}
        \le C \norm{\omega_0}_{L^\iny}.
\end{align*}

We next estimate $\abs{BM_1}$. Because 
\begin{align*}
    B
        = \matrix
            {\prt_1 \brac{\mu_{rh} \prt_1 \Cal{F}_2} * \omega_\eps
                & \prt_2  \brac{\mu_{rh} \prt_1 \Cal{F}_2} * \omega_\eps}
            {\prt_1 \brac{\mu_{rh} \prt_2  \Cal{F}_2} * \omega_\eps
                & \prt_2  \brac{\mu_{rh} \prt_2  \Cal{F}_2} * \omega_\eps}
\end{align*}
we have 
\begin{align*} 
    BM_1
        = \matrix{F_1}{F_2}
        := \matrix
            {(\prt_1 \brac{\mu_{rh} \prt_1 \Cal{F}_2} * \omega_\eps) Y^1_\eps
                + (\prt_2  \brac{\mu_{rh} \prt_1 \Cal{F}_2} * \omega_\eps)Y^2_\eps}
            {(\prt_1 \brac{\mu_{rh} \prt_2  \Cal{F}_2} * \omega_\eps) Y^1_\eps
                + (\prt_2  \brac{\mu_{rh} \prt_2 \Cal{F}_2} * \omega_\eps) Y^2_\eps}.
\end{align*}
We now decompose $F_1$ and $F_2$ into two parts as $F_k = d_k + e_k$, where
\begin{align*}
    d_k
        &= \sum_{j = 1}^2 (\prt_j \brac{\mu_{rh} \prt_k \Cal{F}_2} * \omega_\eps)
                    Y^j_\eps
        - \prt_j \brac{\mu_{rh} \prt_k \Cal{F}_2} * (\omega_\eps Y^j_\eps), \\ 
    e_k
        &= \prt_1 \brac{\mu_{rh} \prt_k \Cal{F}_2} * (\omega_\eps Y^1_\eps)
        + \prt_2  \brac{\mu_{rh} \prt_k \Cal{F}_2} * (\omega_\eps Y^2_\eps) = \dv \bigpr{\mu_{rh} \prt_k \Cal{F}_2 * (\omega_\eps Y_\eps)}.
\end{align*} 

By \cref{L:IntCalBound} (noting that $\dv (\omega_\eps Y_\eps) = \omega_\eps \dv Y_\eps + Y_\eps \cdot \grad \omega_\eps \in C^\al$),
\begin{align*}
    \begin{split}
        \sum_{k = 1, 2} \abs{\lim_{h \to 0} d_k}
            &\le 2\abs{\lim_{h \to 0} \int_{\R^2}
                \nabla \brac{\mu_{rh} \nabla \Cal{F}_2}(x - y) (Y_\eps(x)-Y_\eps(y))
                    \omega_\eps(y) \, dy}\\ 
            &\le C \al^{-1} \norm{Y_\eps(t)}_{C^\alpha} \norm{\omega_\eps(t)}_{L^\iny} r^\alpha
            \le C \al^{-1} \norm{Y_\eps(t)}_{C^\alpha} \norm{\omega_0}_{L^\iny} r^\alpha
    \end{split}
\end{align*}
and
\begin{align} \label{e:ekBound}
    \begin{split}
        \sum_{k = 1, 2} \abs{\lim_{h \to 0} e_k}
            &\le 2 \abs{\lim_{h \to 0} \int_{\R^2} \brac{\mu_{rh}\grad \Cal{F}_2}(x - y)
                \dv (\omega_\eps Y_\eps)(y) \, dy} \\
        &\le C \al^{-1}\norm{\dv (\omega_\eps Y_\eps)(t)}_{C^{\alpha - 1}}  r^\alpha.
    \end{split}
\end{align}
Thus,
\begin{align}
    \lim_{h \to 0} \abs{B}
        &\le C \al^{-1} \frac{P_1(Y_\eps)}{\abs{Y_\eps}^4}
                    \pr{
                        \norm{Y_\eps}_{C^\alpha} \norm{\omega_0}_{L^\iny}
                        + \norm{\dv (\omega_\eps Y_\eps)}_{C^{\alpha - 1}}}
                        r^\alpha
                + C \frac{P_2(Y_\eps)}{\abs{Y_\eps}^2}\norm{\omega_0}_{L^\iny}.
\end{align}

Both sides of the inequality above are functions of $t$ and $x$.
By \cref{e:YepsBelow,e:Y0Below},
\begin{align*}
	\abs{Y_\eps(t, x)}
		\ge I(\Cal{Y}_0)e^{-\int_0^t V_{\epsilon}(s) \, ds}.
\end{align*}
From this, combined with \cref{e:YepsLinfAbove}, we conclude that
\begin{align}\label{e:BFinalEst}
	\begin{split}
	&\sup_{x \in \R^2} \lim_{h \to 0} \abs{B(t, x)} \\
        &\quad
        \le C \al^{-1} \norm{Y_0}_{L^\iny} e^{a_0 \int_0^t V_\eps(s) \, ds}
        	\pr{
                    \pr{
                        \norm{Y_\eps}_{C^\alpha} \norm{\omega_0}_{L^\iny}
                        + \norm{\dv (\omega_\eps Y_\eps)}_{C^{\alpha - 1}}}
                        r^\alpha
                + \norm{\omega_0}_{L^\iny}
                },
    \end{split}
\end{align}
where $a_0 = 9$, since $P_1$ is of degree $5$ by \cref{L:SerfatiLemma1}.

From the estimates in \cref{e:YRCalphaBound}, \cref{e:OnemarKomegaBound}, \cref{e:BFinalEst}, and \cref{L:RegularitywY}, which apply uniformly over all elements of $\Cal{Y}_0$, we conclude that
\begin{align}\label{e:VepsBound}
    \begin{split}
        V_\eps(t)
            &\le 
            C (1 - \log r) \norm{\omega_0}_{L^1 \cap L^\iny}
            	+ \sup_{Y_0 \in \Cal{Y}_0} \sup_{x \in \R^2} \lim_{h \to 0} \abs{B(t, x)} \\
            &\le 
            C(\omega_0) (1 - \log r)
                + \frac{C_\al}{\al} (1 + t) e^{(a_0 + 2) \int_0^t V_\eps(s) \, ds} r^\al
                + C_\al,
    \end{split}
\end{align}
where $C_{\al}$ is defined in \cref{e:Calpha}.

\begin{remark}
Observe how, in contrast to the proof of \cref{T:A} in \cref{S:MatrixA}, we had no need of a partition of unity when bounding $\grad u$, since the regularity of $\grad u$ was not at issue, only a bound on the value of $\abs{\grad u(t, x)}$.
\end{remark}

\subsection{Closing the estimates using Gronwall's lemma}\label{S:ClosingWithGronwall}
Now choose
\begin{align}\label{e:rChoice}
    r
        = \exp \pr{- C' \int_0^t V_\eps(s) \, ds},
\end{align}
delaying the choice of $C'$ for the moment. Then,
\begin{align*}
    1 - \log r
        &\le 1 + C' \int_0^t V_\eps(s) \, ds, \quad
    r^\al
        \le \exp \pr{-C' \al \int_0^t V_\eps(s) \, ds}.
\end{align*}
Returning to \cref{e:VepsBound}, then, these bounds on $1 - \log r$ and $r^\al$ yield the estimate,
\begin{align*}
    V_\eps(t)
        &\le C(\omega_0) + C' C(\omega_0)
            \int_0^t V_\eps(s) \, ds
                + \frac{C_\al}{\al} (1 + t)
                \exp \pr{(a_0 + 2 - \al C') \int_0^t V_\eps(s) \, ds} \\
        &\le \frac{C_\al}{\al} (1 + t) + \frac{C(\omega_0)}{\al}
            \int_0^t V_\eps(s) \, ds
\end{align*}
as long as we set $C' = (a_0 + 2)/\al$

By \cref{L:Gronwall}, we conclude that
\begin{align*}
    \norm{\grad u_\eps(t)}_{L^\iny}
        \le V_\eps(t)
        \le \frac{C_\al}{\al} (1 + t) e^{C(\omega_0) \al^{-1} t}.
\end{align*}

If $\al > 1/2$, we can apply the above bound with $1/2$ in place of $\al$, eliminating the factor of $(1 - \al)^{-1}$ that appear in $C_\al$. This gives
\begin{align}\label{e:graduepsLinffBound}
    \norm{\grad u_\eps(t)}_{L^\iny}
        \le V_\eps(t)
        \le \frac{c_\al}{\al} (1 + t) e^{C(\omega_0) \al^{-1} t}
        \le \frac{c_\al}{\al} e^{C(\omega_0) \al^{-1} t}.
\end{align}
The final inequality is obtained by increasing the value of the constant in the exponent (in a manner that is independent of $\al$.) We do this again, below.

Then
\begin{align*}
    \int_0^t V_\eps(s) \, ds
        < \frac{c_\al}{C(\omega_0)} e^{C(\omega_0) \al^{-1} t}
        = c_\al e^{C(\omega_0) \al^{-1} t}
\end{align*}
so by virtue of \cref{e:YRCalphaBound},
\begin{align}\label{e:YRepsBounds}
    \begin{split}
    & \norm{\Cal{Y}_\eps(t)}_{C^\alpha} 
        \le C_\al \exp \pr{c_\al e^{C(\omega_0) \al^{-1} t}}.
    \end{split}
\end{align}
It follows from \cref{e:YgraduepsBound} that
\begin{align*}
	\norm{\Cal{Y}_\eps \cdot \grad u_\eps(t)}_{C^\al}
        	\le C_\al \al^{-1} \exp \pr{c_\al e^{C(\omega_0) \al^{-1} t}}.
\end{align*}

This gives, once we take $\eps \to 0$ in the next subsection, the estimates in \cref{e:MainBoundsgradu,e:MainBoundsY,e:YgraduBound}.
Similarly, \cref{e:omegaYdivBound} follows from \cref{L:RegularitywY}; \cref{e:gradetaMainBound} follows from \cref{e:gradetaBound,e:gradetaInvBound}; and \cref{e:IYBound} follows from \cref{e:YepsBelow}.

Finally, \cref{e:YdivBound} follows from $\cref{e:CalphaFacts}_1$ applied to $\cref{e:divomegaY}_1$. Here, though, we can absorb the constant $\al c_\al = C(\omega_0, \Cal{Y}_0)$ into the exponent without introducing an additional dependence of the constants on $\al$.

\subsection{Convergence of approximate solutions}\label{S:Convergence}

\noindent That the approximate solutions $(u_\eps)$ converge to the solution $u$ for bounded initial vorticity is by now classical (see Section 8.2 of \cite{MB2002}, for instance). It remains to show, however, that in the limit as $\eps \to 0$, $Y_\eps \to Y$ in such a way that all the estimates in \crefrange{e:MainBoundsgradu}{e:IYBound} hold. This is done by Chemin on pages 105-106 of \cite{C1998}; we highlight here, only the role that assuming $\dv \Cal{Y}_0 \in C^\al$ plays in the convergence argument.

Chemin first establishes that the sequence of flow maps (and inverse flow maps) converge in the sense that  $\eta_\eps - \eta \to 0$ in $L^\iny([0, T] \times \R^2)$ and, similarly, that $\eta_\eps^{-1} - \eta^{-1} \to 0$ in $L^\iny([0, T] \times \R^2)$. Hence, by interpolation, $\eta_\eps - \eta \to 0$ in $L^\iny(0, T; C^\beta(\R^2))$ for all $\beta < 1$ because $\eta_\eps \in L^\iny(0, T; Lip(\mathbb{R}^{d}))$ uniformly in $\epsilon$.

We can write \cref{e:Yeps} as
\begin{align*}
	Y_0 \cdot \grad \eta_\eps
    	= Y_\eps \circ \eta_\eps.
\end{align*}
By \cref{e:CalphaFacts}$_2$ and \cref{e:YRepsBounds}, then, $Y_0 \cdot \grad \eta_\eps$ is uniformly bounded in $L^\iny(0, T; C^\al(\R^2))$. But $C^\al(\R^2)$ is compactly embedded in $C^\beta(\R^2)$ for all $\beta < \al$ so a subsequence of $(Y_0 \cdot \grad \eta_\eps)$ converges in $L^\iny(0, T; C^\beta(\R^2))$ to some $f$ for all $\beta < \al$, and it is easy to see that $f \in L^\iny(0, T; C^\al(\R^2))$.

To show that $f = Y_0 \cdot \grad \eta$, we need only show convergence of $Y_0 \cdot \grad \eta_\eps \to Y_0 \cdot \grad \eta$ in some weaker sense. To do this, observe that 
\begin{align*}
    (Y_0 \cdot \grad \eta_\eps)^j
    	= Y_0 \cdot \grad \eta_\eps^j
        = \dv(\eta_\eps^j Y_0) - \eta_\eps^j \dv Y_0.
\end{align*}
But $\eta_\eps - \eta \to 0$ in $L^\iny(0, T; C^\beta(\R^2))$ for all $\beta < 1$ as we showed above so $\eta_\eps^j Y_0 - \eta^j Y_0 \to 0$ in $L^\iny(0, T; C^\al(\R^2))$ and $\eta_\eps^j \dv Y_0 - \eta^j \dv Y_0 \to 0$ in $L^\iny(0, T; C^\al(\R^2))$. (Here, we used $\dv Y_0 \in C^\al$.) By the definition of negative \Holder spaces in \cref{D:HolderSpaces} it follows that $Y_0 \cdot \grad \eta_\eps \to Y_0 \cdot \grad \eta$ in $L^\iny(0, T; C^{\al - 1}(\R^2))$.
Hence, $f = Y_0 \cdot \grad \eta$, so we can conclude that $Y_0 \cdot \grad \eta \in L^\iny(0, T; C^\al(\R^2))$ and $Y_0 \cdot \grad \eta_\eps \to Y_0 \cdot \grad \eta$ in $L^\iny(0, T; C^\beta(\R^2))$ for all $\beta < \al$.

Then, since
$
    Y_\eps = (Y_0 \cdot \grad \eta_\eps) \circ \eta_\eps^{-1}
$
and
$
    Y = (Y_0 \cdot \grad \eta) \circ \eta^{-1}
$
(see \cref{e:PushForward,e:Yeps}),
we have,
\begin{align*}
    \begin{split}
    \norm{Y_\eps - Y}_{L^\iny}
        &\le
            \norm{(Y_0 \cdot \grad \eta_\eps) \circ \eta_\eps^{-1}
                - (Y_0 \cdot \grad \eta_\eps) \circ \eta^{-1}}_{L^\iny} \\
            &\qquad
            +
                \norm{(Y_0 \cdot \grad \eta_\eps) \circ \eta^{-1}
                - (Y_0 \cdot \grad \eta) \circ \eta^{-1}}_{L^\iny} \\
        &\le
            \norm{Y_0 \cdot \grad \eta_\eps}_{C^\al}
             \smallnorm{\eta_\eps^{-1} - \eta^{-1}}_{L^\iny}^\al
            +
                \norm{Y_0 \cdot \grad \eta_\eps
                - Y_0 \cdot \grad \eta}_{L^\iny} \\
		&\to 0 \text{ as } \eps \to 0,
    \end{split}
\end{align*}
where we used \cref{e:CalphaFacts}$_1$.
Here the $L^\iny$ norms are over $[0, T] \times \R^2$ for any fixed $T > 0$. Arguing as for $Y_0 \cdot \grad \eta$, it also follows that $Y \in L^\iny(0, T; C^\al(\R^2))$ and that the bound on $Y(t)$ in \cref{e:MainBoundsY} holds. Then \cref{e:gradetaMainBound} follows from \cref{e:MainBoundsgradu} as in \cref{e:gradetaBound,e:gradetaInvBound}.

Also,
\begin{align*}
    (Y_\eps \cdot \grad u_\eps)^j
        = \dv(u_\eps^j Y_\eps) - u_\eps^j \dv Y_\eps,
\end{align*}
and given that we now know that $Y_\eps \to Y$ in $C^\beta(\R^2)$ for all $\beta < \al$ with $Y \in C^\al(\R^2)$, \cref{e:YgraduBound} can be proved much the way we proved the convergence of $Y_0 \cdot \grad \eta_\eps \to Y_0 \cdot \grad \eta$, above (taking advantage of \cref{e:YdivBound}, and again using $\dv Y_0 \in C^\al$).

The proofs of the other bounds in \crefrange{e:MainBoundsgradu}{e:IYBound}, which we suppress, follow much the same course as the bounds above. This completes the proof of \cref{T:FamilyVelOmega2D} by Serfati's approach.

\Ignore{ 
\begin{remark}\label{R:WeakerAssumption}
    Had we only assumed that $\dv Y_0 \in C^{\al'}(\R^2)$ for some $\al' \in (0, \al]$
    then the argument above that showed
    $Y_0 \cdot \grad \eta_\eps \to Y_0 \cdot \grad \eta$ in
    $L^\iny(0, T; C^{\al - 1}(\R^2))$ would yield
    $Y_0 \cdot \grad \eta_\eps \to Y_0 \cdot \grad \eta$
    in $L^\iny(0, T; C^{\al' - 1}(\R^2))$. This would be sufficient to conclude
    that $f = Y_0 \cdot \grad \eta$, and the proof would proceed unchanged.
    We could also have established \cref{e:ekBound} under the weaker assumption
    that $\dv Y_0 \in C^{\al'}(\R^2)$,
    because $f \in C^{\al'}(\R^d)$ would have sufficed as an
    assumption in \cref{L:IntCalBound}.
\end{remark}
} 

\section{
Propagation of striated regularity of vorticity in higher dimensions} 
\label{S:3DOutline}

We outline the changes that are needed to the proof of \cref{T:FamilyVelOmega2D} to obtain \noindent\cref{T:FamilyVelOmega3D}.

\medskip

\noindent \cref{S:Transport}: The transport equations involving vorticity are dimension-dependent. Vorticity will remain in $L^\iny$ only for short time because of vortex stretching, which will ultimately limit us to a short-time result. Also, we use the transport of $\dv(\Omega^j_k Y)$ for all $j, k$, in place of $\dv(\omega_\eps Y_\eps)$, though this also will apply only for short time. This is done as in \cite{Danchin1999}.

\medskip

\noindent \cref{S:EstimationGraduAndY}:
We define
\begin{align}\label{e:V3D}
	V_\eps(t)
		:= \norm{\Omega_\eps(t)}_{L^\iny}
			+ \max_{1 \le i, j, k \le d}
				\norm{\PV \int (\prt_i K^k_d)(x - y) \Omega^j_k(y) \, dy}_{L^\iny}
\end{align}
to control $\norm{\grad u_\eps(t)}_{L^\iny}$. (We suppress the $\eps$ subscript that should appear on $\Omega^j_k$ to avoid notational clutter; also notice that there is no sum over $k$.)
The estimates of $\norm{\nabla \eta_\eps(t)}_{L^\iny}$ and $\norm{\nabla \eta^{-1}_\eps(t)}_{L^\iny}$ in \cref{e:gradetaBound,e:gradetaInvBound} are unchanged.

\medskip

\noindent \cref{S:EstYR}:
The estimates of $\norm{Y_\eps}_{L^\iny}$ in \cref{e:YepsLinfAbove} and the bound from below on $\abs{Y(t, x)}$ in \cref{e:YepsBelow} are unchanged. The bound on $\norm{Y_\eps}_{C^\al}$ is derived as in 2D, though now the vortex stretching term in \cref{e:VorticityE} complicates matters. The resolution of this issue is involved, but is handled as in \cite{GamblinSaintRaymond1995,Danchin1999,Fanelli2012}). See, in particular, Sections 4.2.4 and 4.3 of \cite{Fanelli2012}, the vortex stretching term being bounded as in (47) of \cite{Fanelli2012}. (Note that Fanelli is bounding, in effect, $Y_\eps \cdot \grad \Omega^j_k$ in $C^{\al - 1}$ for all $j$, $k$ rather than $\dv (\Omega^j_k Y_\eps)$, but the two are related by his Lemma 4.5.) This yields bounds of the form, 
\begin{align}\label{e:YZCalphaBound3D}
    &\smallnorm{Y_\eps(t)}_{C^\alpha}
        \le C_\al (1 + F(t))
            \exp \pr{2 \int_0^t V_\eps(s) \, ds}.
\end{align}
Here, $F(t)$ is a factor, due to the vortex stretching term, that increases in time in a manner that ultimately prevents Gronwall's inequality from being applied globally in time. (See (49) of \cite{Fanelli2012}.)

\medskip

\noindent \cref{S:Refinedgradueps}: 
Fix $t$, $x$. Let $Y^{(1)}_0, \dots, Y^{(d - 1)}_0 \in \Cal{Y}_0$ be such that
\begin{align*}
	\smallabs{Y^{(1)}_0(x)}, \dots, \smallabs{Y^{(d - 1)}_0(x)}, \quad
		\abs{\wedge_{i < d} Y^{(i)}_0(x)}
		> I(\Cal{Y}_0).
\end{align*}
Let $Y^{(1)}_\eps(t), \dots, Y^{(d - 1)}_\eps(t)$ be the pushforwards of $Y^{(1)}_0, \dots, Y^{(d - 1)}_0$.
Let $W_\eps = \wedge_{i < d} Y^{(i)}$. From the proof of Proposition 4.1 of \cite{Danchin1999}, we have
\begin{align*}
	\prt_t W_\eps + u \cdot \grad W_\eps = - (\grad u)^T W_\eps.
\end{align*}

Examining the estimate that led to \cref{e:YepsBelow}, we see that that argument works just as well for estimating $W_\eps$ from below. This gives
\begin{align*} 
    \abs{W_\eps(t, \eta_\eps(t, x))}
        \ge \abs{W_0(x)} e^{-\int_0^t V_{\epsilon}(s) \, ds}.
\end{align*}

The next significant difference between the $d = 2$ and $d \ge 3$ proofs lies in the refined estimate of $\grad u_\eps$ appearing in \cref{S:Refinedgradueps}, in particular, bounding the matrix,
\begin{align*}
    B
		:= \grad \brac{\mu_{rh} \grad \Cal{F}_d} * \Omega^j_k,
\end{align*}
which arises from the last term in \cref{e:V3D}.
We now choose the $d \times d$ matrix,
\begin{align*}
	M
		=
		\begin{pmatrix}
			Y^{(1)}_\eps & \dots & Y^{(d - 1)}_\eps &W_\eps
		\end{pmatrix}.	
\end{align*}
Because the last column of $M$ is equal to the last column of its cofactor matrix, we can apply \cref{L:SerfatiLemma1}. Then the estimates in \cref{e:YZCalphaBound3D} allow us to bound $\abs{B M_1}, \dots, \abs{B M_{d - 1}}$ just as we did $\abs{B M_1}$ in 2D. The value of the constant $a_0$ in \cref{e:BFinalEst} becomes $4d + 1$, because $P_1$ is of degree $4d - 3$ and $\det M = \abs{W_\eps}^2$, but this does not affect the argument.

\medskip

\noindent \cref{S:ClosingWithGronwall}: The presence of $F(t)$ in \cref{e:YZCalphaBound3D} means that the bound on $V_\eps(t)$ can only be closed for finite time.

\medskip

\noindent \cref{S:Convergence}:
Unlike in 2D, where the existence of a unique solution is assured merely by $\omega_0$ lying in $L^1 \cap L^\iny$ (by Yudovich \cite{Y1963}), existence has to be established using the sequence of approximate solutions. This can be done as in \cite{GamblinSaintRaymond1995,Danchin1999,Fanelli2012}. The proofs of the bounds in \crefrange{e:MainBoundsgradu}{e:IYBound} are unchanged, however, once we have convergence of the flow map and inverse flow map.

%
%
\appendix

%
%
\section{On transport equation estimates}\label{A:TransportEstimates}

\noindent Together, \cref{L:f0,L:ForY0} justify our use of strong transport equations in obtaining estimates in the $C^\al$-norm of the transported and pushed-forward quantities. First, the initial data is mollified using a mollification parameter $\delta$ independent of $\eps$, the strong transport equation estimates are made, then $\delta$ is taken to zero. This is all while $\eps$ is held fixed. \cref{L:f0} is used to obtain the $C^\al$-bound on $\dv Y_\eps(t)$ (leading to \cref{e:YdivBound}), while \cref{L:ForY0} is used to obtain the $C^\al$-bounds on the vector fields, $Y_\eps(t)$ and $Y_\eps \cdot \grad u_\eps(t)$.

The proofs of \cref{L:f0,L:ForY0}, which are left to the reader, employ only \cref{e:CalphaFacts}$_{1, 2}$, the boundedness of $\grad \eta_\eps^{-1}(t)$ in $L^\iny$ over time (for fixed $\eps$), and the convergence in $C^\al$ of a mollified function to the function itself. 

\noindent \begin{lemma}\label{L:f0}
    For $f_0 \in C^\al$ and $\eta_\eps^{-1} \in L^\iny(0, T; Lip(\mathbb{R}^{d}))$, let
    \begin{align*}
        f(t, x)
            &:= f_0(\eta_\eps^{-1}(t, x)), \\
        f^{(\delta)}(t, x)
            &= (\rho_\delta * f_0)(\eta_\eps^{-1}(t, x))
    \end{align*}
    for $\delta > 0$.
    Then
    \begin{align*}
        \smallnorm{f^{(\delta)} - f}_{L^\iny([0, T]; C^\al)}
            \to 0 \text{ as } \delta \to 0.
    \end{align*}
\end{lemma}

\begin{lemma}\label{L:ForY0}
    Let $Y_\eps$ be as in \cref{e:Yeps}, so that
    \begin{align*}
        Y_\eps(t, \eta_\eps(t, x))
            = Y_0(x) \cdot \grad \eta_\eps(t, x).
    \end{align*}    
    Define $Y_\eps^{(\delta)}$ by
    \begin{align*}
        Y_\eps^{(\delta)}(t, \eta_\eps(t, x))
            = (\rho_\delta * Y_0)(x) \cdot \grad \eta_\eps(t, x).
    \end{align*}
    Then
    \begin{align*}
        \smallnorm{Y_\eps^{(\delta)} - Y_\eps}_{L^\iny([0, T]; C^\al)}
            \to 0 \text{ as } \delta \to 0.
    \end{align*}
\end{lemma}

\cref{P:Holder estimate of transport equation} provides the  \Holder bound employed in the proof of  \cref{L:RegularitywY}.

\begin{prop}\label{P:Holder estimate of transport equation}
Let $f_0 \in C^{\al - 1}(\R^d)$, $d \ge 1$, $\al \in (0, 1)$ and suppose that $f$ is $f_0$ transported by the flow map, $\eta$, for the divergence-free velocity field, $v$, which we assume is Lipschitz continuous with a Lipschitz constant that is uniform over the time interval $[0, T]$. Then
\begin{align*}
	\norm{f(t)}_{C^{\al - 1}}
		\le C e^{\int_0^t \norm{\grad v(s)}_{L^\iny} \, ds}
			\norm{f_0}_{C^{\al - 1}}.
\end{align*}
\end{prop}

\begin{proof}
We first note that $\eta, \eta^{-1} \in L^\iny(0, T; Lip(\mathbb{R}^{d}))$ and, moreover, that
\begin{align}\label{e:gradvUpperLower}
	e^{- \int_0^t \norm{\grad v(s)}_{L^\iny} \, ds}
		\le \norm{\grad \eta(t)}_{L^\iny}, \norm{\grad \eta^{-1}(t)}_{L^\iny}
		\le e^{\int_0^t \norm{\grad v(s)}_{L^\iny} \, ds}.
\end{align}

\cref{L:BCD276}, below, gives
\begin{align*}
	\norm{f(t)}_{C^{\al - 1}}
		&= \norm{f(t)}_{B^{\al - 1}_{\iny, \iny}}
		\le C \sup_{\phi \in Q^{1 - \al}_{1, 1}} \innp{f(t), \phi},
\end{align*}
where $Q^{1 - \al}_{1, 1}$ is the set of all Schwartz class functions lying in $B^{1 - \al}_{1, 1}$ having norm $ \le 1$.

By the density of $C^\iny(\R^d)$ in $C^{\al - 1}(\R^d)$ it is sufficient to assume that $f_0 \in C^\iny(\R^d)$.

Let $\phi \in Q^{1 - \al}_{1, 1}$ be such that $\norm{f(t)}_{C^{\al - 1}} \le C \innp{f(t), \phi}$, as guaranteed by \cref{L:BCD276}. Then
\begin{align}\label{e:innpftphi}
	\begin{split}
	\norm{f(t)}_{C^{\al - 1}}
		&\le C \innp{f(t), \phi}
		= C \innp{f_0 \circ \eta^{-1}(t), \phi}
		= C \int_{\R^d} f_0(\eta^{-1}(t, x)) \phi(x) \, dx \\
		&= C \int_{\R^d} f_0(y) \phi(\eta(t, y)) \abs{\det \grad \eta(t, y)} \, dy
		= C \innp{f_0, \phi \circ \eta(t) \abs{\det \grad \eta(t)}} \\
		&= C \innp{f_0, \phi \circ \eta(t)}
		\le C \norm{f_0}_{C^{\al - 1}}
			\norm{\phi \circ \eta(t)}_{B^{1 - \al}_{1, 1}}
	\end{split}
\end{align}
($\det \grad \eta(t, y) = 1$ since $v$ is divergence-free). Here, we again used \cref{L:BCD276}.

Using Theorem 2.36 of \cite{BahouriCheminDanchin2011}, we have
\begin{align*}
	\norm{\phi \circ \eta(t)}_{B^{1 - \al}_{1, 1}}
			&\le C \norm{\frac{\norm{\tau_{-y}
				(\phi \circ \eta(t)) - \phi \circ \eta(t)}_{L^1(\R^d)}}
				{\abs{y}^{1 - \al}}}_{L^1(\R^d; \frac{dy}{\abs{y}^d})} \\
			&= C \int_{\R^d}
				\frac{1}{\abs{y}^{d+1 - \al}}
				\int_{\R^d}
				\abs{\phi \circ \eta(t, x + y) - \phi \circ \eta(t, x)}
					\, dx \, dy \\
			&= C \int_{\R^d}
				\int_{\R^d}
					\frac{\abs{\phi (\eta_t(x + y)) - \phi (\eta_t(x))}}
						{\abs{y}^{d+1 - \al}}
					\, dy \, dx.
\end{align*}
Here we switched to the notation $\eta_t(x)$ for $\eta(t, x)$.
Making the change of variables, $\eta_t(x + y) = \eta_t(x) + z$, which we note induces a $C^1$-diffeomorphism of $\R^d$ with Jacobian $\abs{\det \grad \eta_t(x + t)} = 1$, we have
\begin{align*}
	\norm{\phi \circ \eta(t)}_{B^{1 - \al}_{1, 1}}
			&\le C \int_{\R^d}
				\int_{\R^d}
					\frac{\abs{\phi (\eta_t(x) + z) - \phi (\eta_t(x))}}
						{\abs{\eta_t^{-1}(\eta_t(x) + z) - x}^{d+1 - \al}}
					\, dz \, dx \\
			&\le C e^{(d+1 - \al)\int_0^t \norm{\grad v(s)}_{L^\iny} \, ds}
				\int_{\R^d}
				\int_{\R^d}
					\frac{\abs{\phi (\eta_t(x) + z) - \phi (\eta_t(x))}}
						{\abs{z}^{d+1 - \al}}
					\, dz \, dx \\
			&\le C e^{\int_0^t \norm{\grad v(s)}_{L^\iny} \, ds}
				\int_{\R^d}
				\frac{1}{\abs{z}^{d+1 - \al}}
				\int_{\R^d}
					\abs{\phi (\eta_t(x) + z) - \phi (\eta_t(x))}
					\, dx \, dz,
\end{align*}
where we used \cref{e:gradvUpperLower}. Making another change of variables, $w = \eta_t(x)$, which also has Jacobian of $1$, we have
\begin{align*}
	\norm{\phi \circ \eta(t)}_{B^{1 - \al}_{1, 1}}
			&\le C e^{\int_0^t \norm{\grad v(s)}_{L^\iny} \, ds}
				\int_{\R^d}
				\frac{1}{\abs{z}^{d+1 - \al}}
				\int_{\R^d}
					\abs{\phi (w + z) - \phi (w)}
					\, dw \, dz \\
			& =C e^{\int_0^t \norm{\grad v(s)}_{L^\iny} \, ds}
				\int_{\R^d}
				\frac{1}{\abs{z}^{1 - \al}}
				\int_{\R^d}
					\abs{\phi (w + z) - \phi (w)}
					\, dw \, \frac{dz}{|z|^{d}} \\		
			&= C e^{\int_0^t \norm{\grad v(s)}_{L^\iny} \, ds}
					\norm{\phi}_{B^{1 - \al}_{1, 1}}
			= C e^{\int_0^t \norm{\grad v(s)}_{L^\iny} \, ds}.
\end{align*}
We conclude that
\begin{align}\label{e:TransportBound}
	\norm{f(t)}_{C^{\al - 1}}
		\le C e^{\int_0^t \norm{\grad v(s)}_{L^\iny} \, ds}
			\norm{f_0}_{C^{\al - 1}}.
\end{align}
\end{proof}

\begin{lemma}\label{L:BCD276}
	Let $p, q \in [1, \iny]$ and $s \in \R$
	and let $p'$, $q'$ be the \Holder conjugates of
	$p$, $q$.
	Define $Q^{-s}_{p', q'}$ to be the set of all
	Schwartz-class
	functions lying in $B^{-s}_{p', q'}$ having norm
	$\le 1$. Then for all $u \in B^s_{p, q}$,
	\begin{align}\label{e:uBpqBound}
		\norm{u}_{B^s_{p, q}}
			\le C \sup_{\phi \in Q^{-s}_{p', q'}}
				\innp{u, \phi}
	\end{align}
	and
	\begin{align}\label{e:uphiBound}
		\abs{\innp{u, \phi}}
			\le C \norm{u}_{B^s_{p, q}}.
	\end{align}
	(Note that $\innp{u, \phi}$ is the pairing between
	$u$ and $\phi$ in the duality between $\Cal{S}$
	and $\Cal{S}'$, $\Cal{S}$ being the class of
	Schwartz-class functions.)
\end{lemma}

\begin{proof}
	The inequality in \cref{e:uBpqBound} is part of
	Proposition 2.76 of \cite{BahouriCheminDanchin2011}.
	For \cref{e:uphiBound}, we have,
	using the notation of \cite{BahouriCheminDanchin2011}, 
	\begin{align*}
		\abs{\innp{u, \phi}}
			&= \abs{\innp{\sum_j \Delta_j u,
				\sum_{j'} \Delta_{j'} \phi}}
			= \abs{\sum_j \sum_{j'}
				\innp{\Delta_j u, \Delta_{j'} \phi}} \\
			&= \abs{\sum_{\abs{j - j'} \le 1}
				\innp{\Delta_j u, \Delta_{j'} \phi}}
			= \abs{L(u, \phi)},
	\end{align*}
	since the supports of the Fourier transforms of
	$\Delta_j u$ and $\Delta_{j'} \phi$ are disjoint
	when $\abs{j - j'} \ge 2$. Here, $L$ is the
	continuous bilinear functional on
	$B^{-s}_{p', q'} \times B^s_{p, q}$ defined in
	Proposition 2.76 of \cite{BahouriCheminDanchin2011},
	and \cref{e:uphiBound} follows from the
	continuity of $L$.
\end{proof}

\bigskip
\begin{center}
	{\sc Acknowledgements}
\end{center}
\noindent H. Bae. was supported by the Research Fund (1.170045.01) of Ulsan National Institute of Science \& Technology.  H.B. would like to thank the Department of Mathematics at UC Riverside for its kind hospitality where part of this work was completed. J.P.K. gratefully acknowledges NSF grants DMS-1212141 and DMS-1009545, and thanks Instituto Nacional de Matem\'{a}tica Pura e Aplicada in Rio de Janeiro, Brazil, where part of this research was performed.

\def\cprime{$'$} \def\polhk#1{\setbox0=\hbox{#1}{\ooalign{\hidewidth
  \lower1.5ex\hbox{`}\hidewidth\crcr\unhbox0}}}

\bibliographystyle{plain}

\begin{thebibliography}{10}

\bibitem{BahouriCheminDanchin2011}
Hajer Bahouri, Jean-Yves Chemin, and Rapha{\"e}l Danchin.
\newblock {\em Fourier analysis and nonlinear partial differential equations},
  volume 343 of {\em Grundlehren der Mathematischen Wissenschaften [Fundamental
  Principles of Mathematical Sciences]}.
\newblock Springer, Heidelberg, 2011.

\bibitem{Chemin1991VortexPatch}
Jean-Yves Chemin.
\newblock Existence globale pour le probl\`eme des poches de tourbillon.
\newblock {\em C. R. Acad. Sci. Paris S\'er. I Math.}, 312(11):803--806, 1991.

\bibitem{Chemin1993Persistance}
Jean-Yves Chemin.
\newblock Persistance de structures g\'eom\'etriques dans les fluides
  incompressibles bidimensionnels.
\newblock {\em Ann. Sci. \'Ecole Norm. Sup. (4)}, 26(4):517--542, 1993.

\bibitem{C1998}
Jean-Yves Chemin.
\newblock {\em Perfect incompressible fluids}, volume~14 of {\em Oxford Lecture
  Series in Mathematics and its Applications}.
\newblock The Clarendon Press Oxford University Press, New York, 1998.
\newblock Translated from the 1995 French original by Isabelle Gallagher and
  Dragos Iftimie.

\bibitem{Chemin2004Handbook}
Jean-Yves Chemin.
\newblock Two-dimensional {E}uler system and the vortex patches problem.
\newblock In {\em Handbook of mathematical fluid dynamics. {V}ol. {III}}, pages
  83--160. North-Holland, Amsterdam, 2004.

\bibitem{Danchin1999}
Rapha{\"e}l Danchin.
\newblock Persistance de structures g\'eom\'etriques et limite non visqueuse
  pour les fluides incompressibles en dimension quelconque.
\newblock {\em Bull. Soc. Math. France}, 127(2):179--227, 1999.

\bibitem{Fanelli2012}
Francesco Fanelli.
\newblock Conservation of geometric structures for non-homogeneous inviscid
  incompressible fluids.
\newblock {\em Comm. Partial Differential Equations}, 37(9):1553--1595, 2012.

\bibitem{GamblinSaintRaymond1995}
Pascal Gamblin and Xavier Saint~Raymond.
\newblock On three-dimensional vortex patches.
\newblock {\em Bull. Soc. Math. France}, 123(3):375--424, 1995.

\bibitem{Gunter1927A}
N.~Gunther.
\newblock On the motion of fluid in a moving container.
\newblock {\em Izvestia Akad. Nauk USSR, Ser. Fiz. Mat.}, 20:1323--1348,
  1503--1532, 1927.

\bibitem{Gunter1927B}
N.~Gunther.
\newblock On the motion of fluid in a moving container.
\newblock {\em Izvestia Akad. Nauk USSR, Ser. Fiz. Mat.}, 21:521--526,735--756,
  1927.

\bibitem{Gunter1928}
N.~Gunther.
\newblock On the motion of fluid in a moving container.
\newblock {\em Izvestia Akad. Nauk USSR, Ser. Fiz. Mat.}, 22:9--30, 1928.

\bibitem{Hadamard1893}
J.~Hadamard.
\newblock Resolution d'une question relative aux determinants.
\newblock {\em Bull. des Sciences Mathematiques}, 17:240--246, 1893.

\bibitem{KNV2007}
A.~Kiselev, F.~Nazarov, and A.~Volberg.
\newblock Global well-posedness for the critical 2{D} dissipative
  quasi-geostrophic equation.
\newblock {\em Invent. Math.}, 167(3):445--453, 2007.

\bibitem{Leray1933}
Jean Leray.
\newblock {\'E}tude de diverses \'equations int\'egrales non lin\'eaires et de
  quelques probl\`emes que pose l'{H}ydrodynamique.
\newblock {\em J. Math. Pures Appl.}, 12:1--82, 1933.

\bibitem{Lichtenstein1925}
Leon Lichtenstein.
\newblock \"{U}ber einige {E}xistenzprobleme der {H}ydrodynamik homogener,
  unzusammendr\"{u}ckbarer, reibungsloser {F}l\"{u}ssigkeiten und die
  {H}elmholtzschen wirbels\"{a}tze.
\newblock {\em Math. Z.}, 23(1):89--154, 1925.

\bibitem{Lichtenstein1927}
Leon Lichtenstein.
\newblock \"{U}ber einige {E}xistenzprobleme der {H}ydrodynamik. {Z}weite
  {A}bhandlung. {N}ichthomogene, unzusammendr\"{u}ckbare, reibungslose
  {F}l\"{u}ssigkeiten.
\newblock {\em Math. Z.}, 26(1):196--323, 1927.

\bibitem{Lichtenstein1928}
Leon Lichtenstein.
\newblock \"{U}ber einige {E}xistenzprobleme der {H}ydrodynamik. {III}.
  {P}ermanente {B}ewegungen einer homogenen, inkompressiblen, z\"{a}hen
  {F}l\"{u}ssigkeit.
\newblock {\em Math. Z.}, 28:387--415, 1928.

\bibitem{Lichtenstein1930}
Leon Lichtenstein.
\newblock \"{U}ber einige {E}xistenzprobleme der {H}ydrodynamik. {IV}.
  {S}tetigkeitss\"{a}tze. {E}ine {B}egr\"{u}ndung der
  {H}elmholtz-{K}irchhoffschen {T}heorie geradliniger {W}irbelf\"{a}den.
\newblock {\em Math. Z.}, 32(1):608--640, 1930.

\bibitem{MB2002}
Andrew~J. Majda and Andrea~L. Bertozzi.
\newblock {\em Vorticity and incompressible flow}, volume~27 of {\em Cambridge
  Texts in Applied Mathematics}.
\newblock Cambridge University Press, Cambridge, 2002.

\bibitem{SerfatiUnpublished3DVortexPatches}
Philippe Serfati.
\newblock Vortex patches et r\'{e}gularit\'{e} stratif\'{e}e.
\newblock {\em Undated preprint, 16 pp.}

\bibitem{SerfatiThesis}
Philippe Serfati.
\newblock Etude math\'{e}matique de flammes infiniment minces en combustion.
  {R}\'{e}sultats de structure et de r\'{e}gularit\'{e} pour l'\'{e}quation
  d'{E}uler incompressible.
\newblock {\em Thesis, Univerisity Paris 6}, 1992.

\bibitem{Serfati3DStratified}
Philippe Serfati.
\newblock R\'egularit\'e stratifi\'ee et \'equation d'{E}uler {$3$}{D} \`a
  temps grand.
\newblock {\em C. R. Acad. Sci. Paris S\'er. I Math.}, 318(10):925--928, 1994.

\bibitem{SerfatiVortexPatch1994}
Philippe Serfati.
\newblock Une preuve directe d'existence globale des vortex patches {$2$}{D}.
\newblock {\em C. R. Acad. Sci. Paris S\'er. I Math.}, 318(6):515--518, 1994.

\bibitem{S1970}
Elias~M. Stein.
\newblock {\em Singular integrals and differentiability properties of
  functions}.
\newblock Princeton Mathematical Series, No. 30. Princeton University Press,
  Princeton, N.J., 1970.

\bibitem{Wolibner1933}
W.~Wolibner.
\newblock Un theor\`eme sur l'existence du mouvement plan d'un fluide parfait,
  homog\`ene, incompressible, pendant un temps infiniment long.
\newblock {\em Mathematische Zeitschrift}, 37(1):698--726, 1933.

\bibitem{Y1963}
V.~I. Yudovich.
\newblock Non-stationary flows of an ideal incompressible fluid.
\newblock {\em \u Z. Vy\v cisl. Mat. i Mat. Fiz.}, 3:1032--1066 (Russian),
  1963.

\end{thebibliography}

\end{document}